\documentclass[12pt]{article}
\newcommand{\editor}{Sudipta Mallik}
\newcommand{\ArticleType}{Research Article}
\newcommand{\ReceivedDate}{Oct 23, 2024}
\newcommand{\RevisedDate}{Feb 19, 2025}
\newcommand{\AcceptedDate}{Feb 21, 2025}
\newcommand{\PublishedDate}{Feb 24, 2025}
\newcommand{\JournalIndex}{Volume 4 (2025), Pages 1--22}
\newcommand{\LastName}{Elder}
\newcommand{\JournalIndexShort}{\LastName\, et al./ American Journal of Combinatorics 4 (2025) 1--22}

\usepackage{fancyhdr}
\usepackage{amsmath,amsthm,amssymb,mathtools,url,graphicx,pdfpages,tikz,rotating}
\usepackage[bookmarks=true,pdfborder={0 0 0}]{hyperref}
\usepackage[affil-it]{authblk}
\usepackage[left=1in,right=1in,top=1.2in, bottom=1in]{geometry}

\usepackage{xcolor,enumitem,subcaption,color,pgf,float,cleveref}
\usepackage[utf8]{inputenc}
\usetikzlibrary{calc}

\newtheorem{theorem}{Theorem}[section]
\newtheorem{lemma}[theorem]{Lemma}
\newtheorem{corollary}[theorem]{Corollary}

\newtheorem{notation}[theorem]{Notation}
\theoremstyle{definition}
\newtheorem{definition}[theorem]{Definition}
\newtheorem{example}[theorem]{Example}

\newtheorem{remark}[theorem]{Remark}
\newtheorem{proposition}[theorem]{Proposition}

\numberwithin{equation}{section}

\newcommand{\PF}{\mathrm{PF}}

\newcommand{\Le}{\mathrm{Left}} 
\newcommand{\R}{\mathrm{Right}}
\newcommand{\B}{\mathrm{B}}
\newcommand{\F}{\mathrm{F}}
\newcommand{\C}{\mathrm{C}}
\newcommand{\Sym}{\mathfrak{S}}
\newcommand{\Tym}{\mathfrak{T}}
\newcommand{\Pref}{\mathrm{Pref}}

\newcommand{\VPF}{\mathrm{VPF}}

\newcommand{\NN}{\mathbb{N}}

\newcommand{\out}{\mathcal{O}}

\begin{document}
\setcounter{page}{1}
\noindent {\color{teal}\bf\large American Journal of Combinatorics} \hfill \ArticleType\\
\JournalIndex

\title{Pullback parking functions}

\author{Jennifer Elder*, Pamela E. Harris, Lybitina Koene, Ilana Lavene,\\Lucy Martinez, and Molly Oldham}


\affil{\normalsize\rm (Communicated by \editor)
\vspace*{-24pt}}
\date{}

{\let\newpage\relax\maketitle}

\begin{abstract}
 We introduce a generalization of parking functions in which cars are limited in their movement backwards and forwards by two nonnegative integer parameters $k$ and $\ell$, respectively. 
        In this setting, there are $n$ spots on a one-way street and $m$ cars attempting to park in those spots, and $1\leq m\leq n$. We let $\alpha=(a_1,a_2,\ldots,a_m)\in[n]^m$ denote the parking preferences for the cars, which enter the street sequentially. Car $i$ drives to their preference $a_i$ and parks there if the spot is available. Otherwise, car $i$ checks up to $k$ spots behind their preference, parking in the first available spot it encounters if any. If no spots are available, or the car reaches the start of the street, then the car returns to its preference and attempts to park in the first spot it encounters among spots  $a_i+1,a_i+2,\ldots,a_i+\ell$. If car $i$ fails to park, then parking ceases. If all cars are able to park given the preferences in $\alpha$, then $\alpha$ is called a $(k,\ell)$-pullback $(m,n)$-parking function. Our main result establishes counts for these parking functions in two ways: counting them based on their final parking outcome (the order in which the cars park on the street), and via a recursive formula. Specializing $\ell=n-1$, our result gives a new formula for the number of $k$-Naples $(m,n)$-parking functions and further specializing $m=n$ recovers a formula for the number of $k$-Naples parking functions given by Christensen et al. The specialization of $k=\ell=1$, gives a formula for the number of vacillating $(m,n)$-parking functions, a generalization of vacillating parking functions studied by Fang et al., and the $m=n$ result answers a problem posed by the authors. 
    We conclude with a few directions for further study.
\end{abstract}

\renewcommand{\thefootnote}{\fnsymbol{footnote}} 
\footnotetext{\hspace*{-22pt} * Corresponding author
Received \ReceivedDate; Revised \RevisedDate; Accepted \AcceptedDate; Published \PublishedDate \\
\copyright\, The author(s). Released under the CC BY 4.0 International License 
}
\renewcommand{\thefootnote}{\arabic{footnote}}

\section{Introduction}
        Consider the following parking scenario. There are $m$ cars attempting to park on a one-way street with $n$ parking spots (with $n\geq m$). 
        The queue of cars has a list of parking spot preferences represented by a tuple, $\alpha = (a_1,a_2,\ldots,a_m)$, where each $a_i\in[n]=\{1,2,\ldots,n\}$. 
        The cars enter the street in sequential order from car $1$ to car $m$, and firstly will attempt to park in their preference. 
        If that space is unoccupied, then the car parks there. Otherwise, the car proceeds forward and parks in the first available spot it encounters (if there is any). 
        We refer to this parking process as the \textit{classical parking rule}. 
        If the list of parking preferences allows all cars to park among the $n$ spots on the street, then we call the list of preferences an \textit{$(m,n)$-parking function}. 
        For example, when $n=m=4$, the preference list $(1,4,3,2)$ allows all of the cars to park as they each prefer a different parking spot. However, when $m=3$ and $n=4$, the preference list $(1,4,4)$ is not a parking function as the third car finds the fourth spot occupied by the second car, and drives forward to the end of the street, unable to park.
        
        Throughout, we let $\PF_{m,n}$ denote the set of parking functions with $m$ cars and $n$ spots, and when the number of cars and spots is clear from context, we simply refer to this set as the set of parking functions.
        Of course, if $m>n$, then, no matter what the cars' preferences are, there will always be a car that fails to park. 
        Hence, whenever $m>n$, $|\PF_{m,n}|=0$. 
        In the case where $m\leq n$, Konheim and Weiss \cite{konheim1966occupancy} established that $|\PF_{m,n}|=(n+1-m)(n+1)^{m-1}$. 
        In particular, if $n=m$, we let $\PF_n=\PF_{n,n}$, and note that the cardinality of this set of parking functions simplifies nicely to $|\PF_{n}|=(n+1)^{n-1}$.
        There are many generalizations of parking functions and most important to our study are $k$-Naples parking functions and $\ell$-interval parking functions.
        
        Introduced by Baumgardner~\cite{BaumgardnerHonorsContract}, and generalized by Christensen, Harris, Jones, Loving, Ramos Rodr\'{i}guez, Rennie, and Rojas Kirby ~\cite{knaple},
        the set of $k$-Naples parking functions sets $m=n$ and extends the classical parking rule by allowing a car that finds its preferred parking spot occupied to first back up to $k$ spaces, checking one spot at a time, for $0\leq k< n$, in its attempt to park. 
        If possible, the car parks in the first available spot among those $k$ spots. If none of the $k$ spaces before its preferred parking spot are available (or the car reaches the start of the street), then the car continues past its preferred spot and parks in the first available spot beyond it. 
        We call this parking rule the \textit{$k$-Naples parking rule}.
        If the parking preference $\alpha$ allows all cars to park using the $k$-Naples parking rule, then we say that $\alpha$ is a \textit{$k$-Naples parking function of length $n$}. We let $\PF_{n}(k)$ denote the set of $k$-Naples parking functions of length $n$.
        If $k=0$, then $\PF_{n}(0)=\PF_n$. 
        Also, when $k=1$, the set $\PF_{n}(1)$ is called the set of \textit{Naples parking functions} of length $n$. 
        Based on these definitions one can observe that $\PF_{n}(k-1)\subseteq \PF_{n}(k)$ for all $0 < k< n$. 
        The set inclusion occurs because every car that can park when allowed to back up $k-1$ spots can also park when allowed to back up $k$ spots. 
        Christensen et al.~established a recursive formula for the number of $k$-Naples parking functions \cite[Theorem 1.1]{knaple}. 
        In our work, we extend the definition of  $k$-Naples parking functions to the case where there are more parking spaces than cars, and we refer to this set as the set of $k$-Naples $(m,n)$-parking functions. We denote this set by $\PF_{m,n}(k)$.
        In Corollary \ref{cor:knaples mnpfs}, we give a recursion for the number of $k$-Naples $(m,n)$-parking functions for all~$m\leq n$ and $0 \leq k < n$. 
        
  Another generalization of classical parking functions called $\ell$-interval $(m,n)$-parking func-tions was introduced by Aguilar-Fraga, Elder, Garcia, Hadaway, Harris, Harry, Hogan, Johnson, Kretschmann, Lawson-Chavanu,   Mart\'inez Mori, Monroe, Qui\~nonez, Tolson III, and Williams II~\cite{aguilarfraga2024intervalellintervalrationalparking}. 
        In this scenario, a car drives to its preference and parks there if the spot is available. Otherwise, if occupied, then the car drives forward, checks up to $\ell$ spots past its preference, and parks in the first available spot it encounters among those $\ell$ spots, if such a spot exists. 
        If all of the cars can park under the $\ell$-interval parking rule, the list of preferences is an $\ell$-interval parking function.
        For example, if $m=n$ and $\ell=0$, then the cars can only park in their preferred parking spot, so the set of 
        $0$-interval parking functions is precisely the set of permutations on $[n]$, which we denote by $\Sym_n$.
        In the case where $\ell=n-1$, then cars are able to traverse the full street in seeking a parking spot, hence the set of $(n-1)$-interval parking functions is exactly the set $\PF_{m,n}$.
        As a concrete example, if $m=n=3$, then $(1,1,1)$ is not a $1$-interval parking function, because when car $3$ attempts to park, the only remaining unoccupied spot on the street is spot $3$, which is more than one spot ahead of its preferred spot.
        Among their results, Aguilar-Fraga et~al.~provide formulas for the number of $\ell$-interval $(m,n)$-parking functions, see \cite[Theorems  3.7 and 3.8]{aguilarfraga2024intervalellintervalrationalparking}. 
        
        Based on the $k$-Naples and $\ell$-interval parking rules, we consider a new parking rule in which cars are restricted in their backward and forward movement whenever they find their preferred parking spot occupied. 
        This is reminiscent of pullback car toys (see Figure \ref{fig:pullback_toy}), which are a popular children's toy car with a winding mechanism used by placing the car on the floor and rolling it backwards, then releasing, the car is propelled forward some distance. Based on this toy, we use the language of pullback parking rule to describe the behavior of cars when they find their preferred parking spot occupied. 
        We now give the technical definition for this parking rule.
        For fixed nonnegative integers $k$ and  $\ell$, we defined the \textit{$(k,\ell)$-pullback parking rule}:
        given a parking preference $\alpha =(a_1,a_2,\ldots,a_m)\in[n]^m$, for each $i\in[m]$, car $i$ drives to its preferred parking spot $a_i$. If it is available, then the car parks there. If the spot is occupied, then the car checks up to $k$ spots behind its preference, parking in the first available spot it encounters (if any). If all of those parking spots are occupied or the car reaches the start of the street, then the car checks up to $\ell$ spots ahead of its preference, parking in the first available spot it encounters (if any). 
        If the car fails to park within those spots, then the car exits the street without parking.
        If all cars can park using the $(k,\ell)$-pullback parking rule, then the preference list is called a 
        \textit{$(k,\ell)$-pullback $(m,n)$-parking function}.
         We let $\PF_{m,n}(k,\ell)$ denote the set of all $(k,\ell)$-pullback $(m,n)$-parking functions.
        
        \begin{figure}[H]
        
            \centering\includegraphics[width=3in]{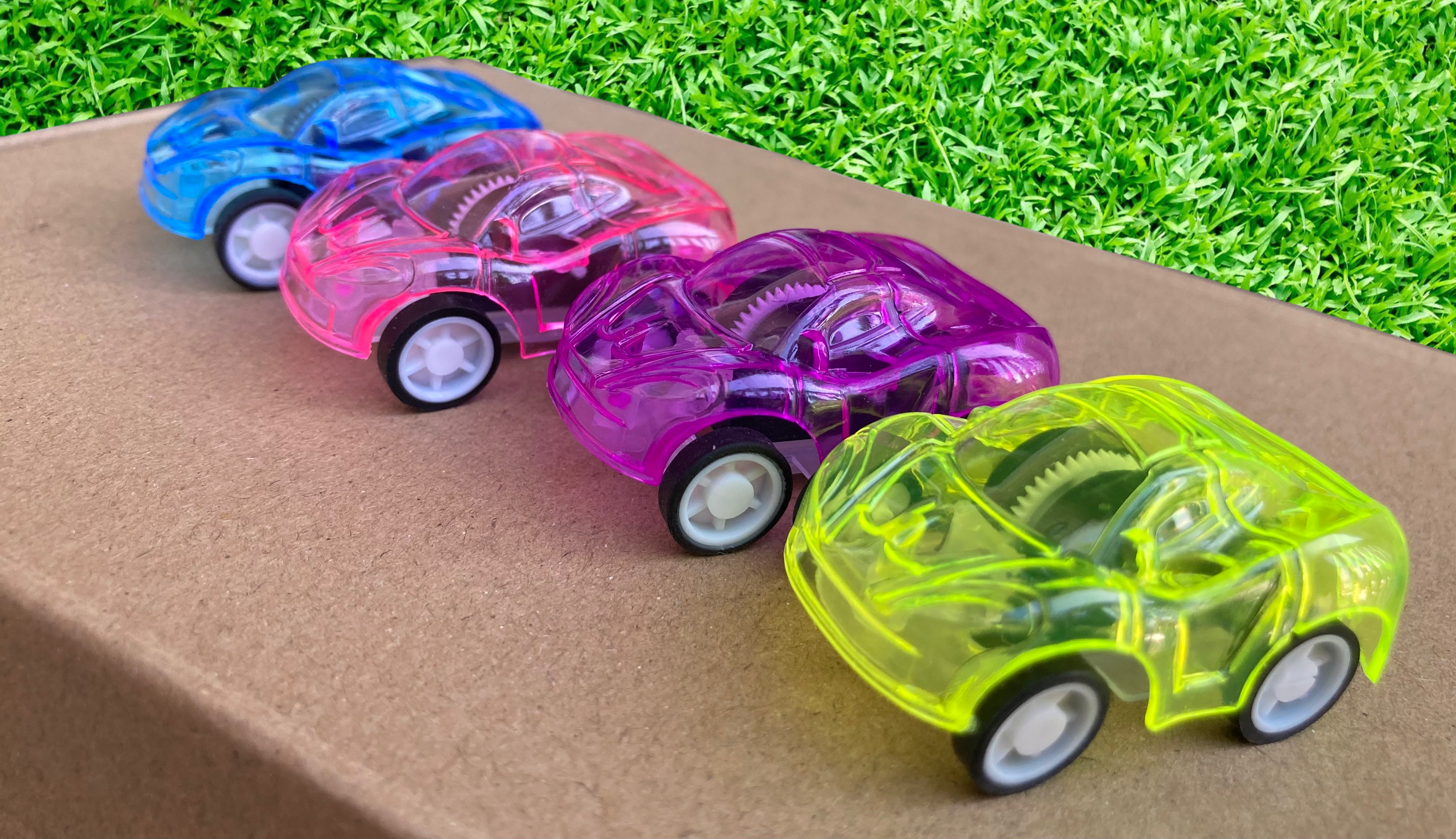}
            
            \caption{Pullback toy cars. Image credit: Lucy Martinez.}
            \label{fig:pullback_toy}
            
        \end{figure}

        In Figure \ref{fig:a}, we illustrate that $(3,2,3,1)\in\PF_{4,5}(1,2)$, since all of the cars are able to park with the given constraints. 
        For the preference 
        $(3,2,2,1)$ using the $(1,2)$-pullback parking rule when $m=4$ and $n=5$, car $4$ finds the first spot occupied by car $3$, and is unable to backup as it has reached the start of the street, so it proceeds forward but finds the two spots past its preference also occupied. 
        Therefore, car $4$ fails to park and  $(3,2,2,1)$ is not a $(1,2)$-pullback $(4,5)$-parking function.

        \begin{figure}[H]
        \centering
        \includegraphics[width=5in]{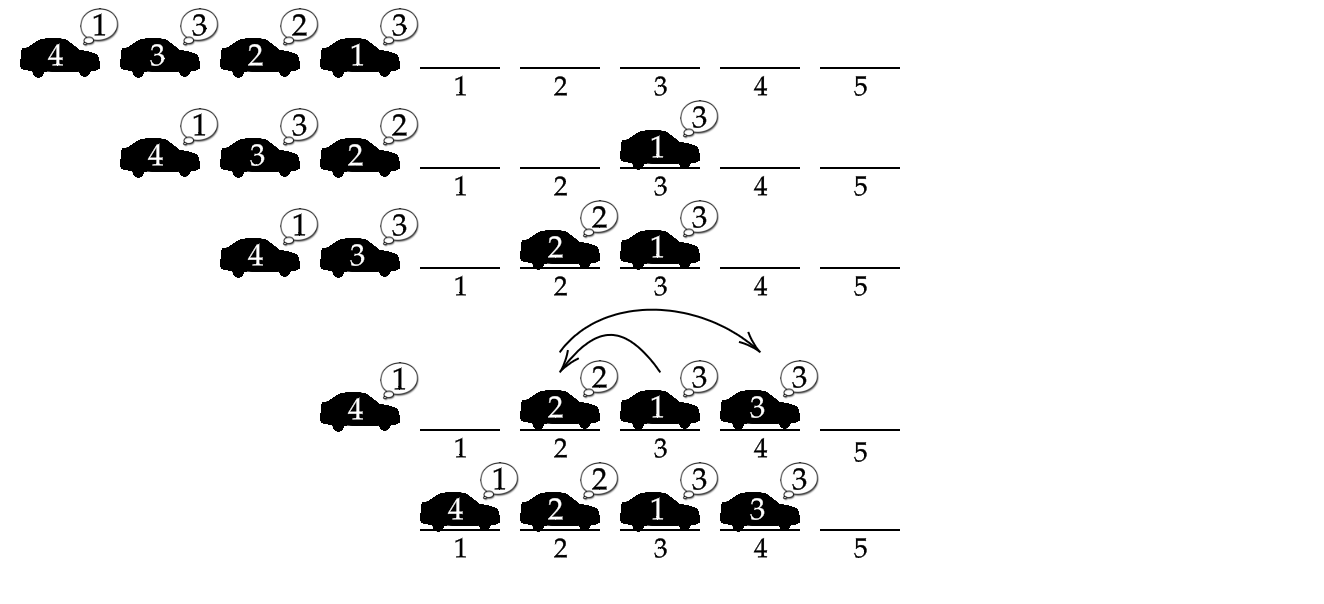}
        \caption{Parking procedure for $\alpha=(3,2,3,1)$ using the $(1,2)$-pullback parking rule.}            
        \label{fig:a}
        \end{figure}
        
        Our major contributions include two formulas to count the number of $(k,\ell)$-pullback $(m,n)$-parking functions for all nonnegative integers $k$ and $\ell$, and all positive integers $1\leq m\leq n$. 
        In Theorem \ref{thm:counting through perms} we use a technique known as ``counting through permutations'', which gives a count for the number of parking functions that park cars in a specified order. The name of this technique comes from the fact that when $m=n$, the order in which the cars park on the street can be described using a permutation in $\Sym_n$. 
        In Theorem \ref{thm:pullback} we provide a purely recursive formula for this count. 
        
        This article is organized as follows. In Section \ref{sec:counting through permutations}, we begin by utilizing the counting parking functions through permutations technique to prove Theorem \ref{thm:counting through perms}, giving our first formula for the number of $(k,\ell)$-pullback $(m,n)$-parking functions. In Corollary \ref{cor:counting through perms}, we give an alternate formula for the number of $k$-Naples $(m,n)$-parking functions.
        In Section \ref{sec:recursive formula}, we recall the formula for the number of $k$-Naples parking functions with the same number of cars and spots \cite[Theorem 1.1]{knaple}, and then prove Theorem \ref{thm:pullback}, which gives a recursive formula for the number of $(k,\ell)$-pullback parking functions with $m$ cars and $n$ spots. Once again, in Corollary \ref{cor:knaples mnpfs}, we give an alternate formula for the number of $k$-Naples $(m,n)$-parking functions.
        We conclude with Section \ref{sec:future}, where we provide directions for future research.

        \begin{remark}
            For code supporting the work in this project we provide the following \href{https://github.com/lybitinakoene/Pullback-Parking-Functions.git}{GitHub repository} \cite{github}. 
        \end{remark}
   
    \section{Counting through permutations}\label{sec:counting through permutations}
        In this section, we recall a technique known in the literature as ``counting through permutations,'' which has been used to count the number of parking preferences parking cars in a specified order.
        To make our approach precise, we begin with a definition. 
        
        \begin{definition}[\cite{countingthroughperms,Spiro}]\label{def:outcome}
            Given a parking function $\alpha=(a_1,a_2,\ldots,a_n)\in[n]^n$, define the \textit{outcome} of $\alpha$ by  $\out (\alpha)=\pi_1\pi_2\cdots\pi_n\in \Sym_n$, where $\pi_i=j$ denotes that car $j$ parked in spot $i$.           
        \end{definition}
        
        \begin{example}
        
            If $\alpha=(1,1,1,2,4,4,5,7)\in\PF_8$,
            then the outcome of $\alpha$ is given by 
            $\out (\alpha)=12345678$.
            If $\alpha=(7,1,5,2,4,1,4,1)\in\PF_8$,
            then the outcome is 
            $\out (\alpha)=24653718$.
        \end{example}
        
        We recall the following result which counts parking functions with a specified outcome.
        
        \begin{proposition} \cite[Proposition 3.3]{countingthroughperms}\label{prop:permutation}
            Let $\sigma = \sigma_1\sigma_2\cdots \sigma_n\in \Sym_n$ be a permutation. 
            The number of parking functions with outcome $\sigma$ is 
            $\prod_{i=1}^n \mathcal{L}(\sigma_i)$,
            where $\mathcal{L}(\sigma_i)$ is the length of the longest subsequence, $\sigma_j,\sigma_{j+1},\ldots, \sigma_i$, of $\sigma$ such that $\sigma_t\leq \sigma_i$ for all $j\leq t\leq i$. 
        \end{proposition}
        
        Proposition \ref{prop:permutation} was used to give the number of parking functions by counting through permutations. We recall this result next.
        
        \begin{corollary}[{\cite[Corollary 3.5]{countingthroughperms},\cite[Exercise 5.49(d,e)]{StanleyECVol2}}]
        
            It follows that the number of parking functions is given by 
            $|\PF_n|=\sum_{\sigma\in \Sym_n}\left(\prod_{i=1}^n \mathcal{L}(\sigma_i)\right)=(n+1)^{n-1}$.
        \end{corollary}
        
        We extend Definition 
        \ref{def:outcome} for the case where there are more parking spots than~cars.
        
        \begin{definition}\cite[Definition 3.1]{aguilarfraga2024intervalellintervalrationalparking}\label{def:set Smn}
            Let $\Sym_{m,n}$ denote the set of permutations of the multiset $\{0,\ldots,0\}\cup[m]$ with $n-m$ zeros.
            Given an $(m,n)$-parking function $\alpha=(a_1,a_2,\ldots,a_m)$, define the \textit{outcome} of $\alpha$ by $\mathcal{O}(\alpha)=\pi_1\pi_2\cdots\pi_n\in \Sym_{m,n}$ where $\pi_i=j\in[m]$ denotes that car $j$ parked in spot $i$, and $\pi_i=0$ indicates that spot $i$ is vacant.
        
        \end{definition}
        
        \begin{definition}
        Let $\pi=\pi_1\pi_2\cdots\pi_n\in \Sym_{m,n}$. Then, for each $i\in[n]$ with $\pi_i>0$, 
            let $\Pref(\pi_i)$ denote the set of preferences of car $\pi_i$ such that it is the $i$th car parked on the street.
        \end{definition}
        
        Our goal is to answer the following question: Given $\pi=\pi_1\pi_2\cdots\pi_n\in \Sym_{m,n}$, how many $(k,\ell)$-pullback $(m,n)$-parking functions, $\alpha\in \PF_{m,n}(k,\ell)$, have outcome $\out (\alpha)=\pi$?
        Before stating this result, we provide the following illustrative example. 
        
        \begin{example}\label{ex:pref_count}
        
            Let $m=8$, $n=11$, $k=1$, $\ell=2$ and consider $\pi = 08134005672$.
            Then we make the following observations:            \begin{itemize}
            
                \item Car $1$ parked in the $3$rd spot on the street. Since it was the first car entering the street, it must have preferred spot $3$. Hence, $|\Pref(\pi_3 = 1)|=|\{3\}|=1$.
                
                \item Car $2$ parked in the $11$th spot on the street. Since there were no cars already parked immediately to the left or right of this spot, car $2$ could have only preferred spot $11$. Hence, $|\Pref(\pi_{11} = 2)|=|\{11\}|=1$.
                
                \item Car $3$ parked in the $4$th spot on the street. Even though car $1$ was already parked immediately to the left of spot $4$, if car $3$ had preferred spot $3$, it would have backed up and parking in the $2$nd spot and would not have ended up in the $4$th spots, so car $3$ could have only preferred spot $4$ and parked there. Hence, $|\Pref(\pi_4 = 3)|=|\{4\}|=1$.
                
                \item Car $4$ parked in the $5$th spot on the street. Notice that cars $1$ and $3$ have  parked to the immediate left of spot $5$. 
                If car $4$ had preferred spot $3$, it would have backed up and parked in the $2$nd spot and would not have ended up in spot $5$. However, car $4$ could have preferred spot $4$ because, since $k=1$, it would have backed up to spot $3$ first and, finding it occupied, then pulled forward parking in spot $5$, which is within its tolerance as $\ell=2$. Car $4$ could have also preferred spot $5$ and parked there. Hence, $|\Pref(\pi_5 = 4)|=|\{4,5\}|=2$.

                \item Car $5$ parked in the $8$th spot on the street. Since there were no cars parked to the immediate left or right of spot $8$, car $5$ could have only preferred spot $8$ and parked there. Hence, $|\Pref(\pi_8 = 5)|=|\{8\}|=1$.
                
                \item Car $6$ parked in the $9$th spot on the street. Car $6$ could not have preferred spot $8$, as it would have backed up and parking in spot $7$ and would not have ended up parking in spot $9$. Thus, car $6$ could only have preferred spot $9$ and parked there. Hence, $|\Pref(\pi_9 = 6)|=|\{9\}|=1$.
                
                \item Car $7$ parked in the $10$th spot on the street. 
                If car $7$ preferred spot $9$ it would find it occupied and since $k=1$, it also would have found spot $8$ occupied. Then moving forward, it would park in spot in $10$.
                If car $7$ preferred spot $11$, finding it occupied it could have  backed into spot $10$. Finally, it could also be that car $7$ preferred spot $10$ and parked there. Hence, $|\Pref(\pi_{10} = 7)|=|\{9,10,11\}|=3$.
                
                \item Car $8$ parked in the $2$nd spot on the street. 
                If car $8$ preferred spot $3$, finding it occupied, it would back into spot $2$. 
                If car $8$ preferred spot $2$, it would park there. Hence, $|\Pref(\pi_2 = 8)|=|\{2,3\}|=2$.
            \end{itemize}
            
            The total number of parking functions, $\alpha = (a_1,a_2,\ldots,a_8)\in\PF_{8,11}$, that park the cars in order $\pi$ must satisfy $a_i \in \Pref(\pi_j = i)$ for each $i\in [m]$. 
            Thus, the total number of these parking functions is given by the product
            $
            \prod_{i=1}^m|\Pref(\pi_j = i)| = 1\cdot 1\cdot 1\cdot 2\cdot 1\cdot 1\cdot 3\cdot 2=12$.
        \end{example}

As Example \ref{ex:pref_count} illustrates, determining the possible preferences of  car $i$ requires knowledge on what cars are parked immediately to the left and to the right of car $i$. 
The intuition here is that a car could have 
\begin{enumerate}
\item found its preferred spot occupied and the car backed into its final parking spot. This would only happen if it preferred a spot to the right of where it parked, and it parked by having checked no more than $k$ spots to the left of its preference; \label{item:case 1}
\item found its preferred spot occupied and the car moved forward into its final parking spot. This would only happen if 
its preference was to the left of where it parked, and 
it had already checked at most $k$
spots to the left of its preference finding all of those spots occupied. Then the car would return to its preference and find its final parking spot within $\ell$ spots to the right of its preference; \label{item:case 2} 
\item found its preferred spot available, and parked there. \label{item:case 3} 
\end{enumerate}
        
        To count the number of pullback parking functions that park the cars in the fixed order $\pi\in \Sym_{m,n}$, requires us to determine the counts for the preferences, which we do using the three cases above. We define this next.

        \begin{definition}\label{def:cases_for_prefs}
            Fix $\pi=\pi_1\pi_2\cdots\pi_n\in \Sym_{m,n}$. 
        Then, for each $i\in[n]$ with $\pi_i>0$, let  
        \begin{itemize}
            \item $\B(\pi_i)$ be the number of preferences for car $\pi_i$ satisfying Case (\ref{item:case 1}), and
            \item $\F(\pi_i)$ be the number of preferences for car $\pi_i$ satisfying Case (\ref{item:case 2}).
        \end{itemize}
        \end{definition}

        Next we define a function which counts cars that arrived prior to car $\pi_i$ and parked contiguously to the immediate right of car $\pi_i$. 

\begin{definition}\label{def:LookRight}
        Fix $\pi=\pi_1\pi_2\cdots\pi_n\in \Sym_{m,n}$. 
        Then, for each $i\in[n]$ with $\pi_i>0$, 
            let $\R(\pi_i)$ be the length of the longest consecutive subsequence, $\pi_{i+1},\pi_{i+2},\ldots,\pi_{i+x}$, such that $0<\pi_t<\pi_i$ for all $i+1\leq t\leq i+x$.
        \end{definition}

Next we define a function which counts cars that arrived prior to car $\pi_i$ and parked contiguously to the immediate left of car $\pi_i$. 

        \begin{definition}\label{def:LookLeft}
Fix $\pi=\pi_1\pi_2\cdots\pi_n\in \Sym_{m,n}$. 
        Then, for each $i\in[n]$ with $\pi_i>0$, 
            let $\Le(\pi_i)$ be the length of the longest consecutive subsequence, $\pi_y,\pi_{y+1},\ldots,\pi_{i-1}$, such that $0<\pi_t<\pi_i$ for all $y\leq t \leq i-1$.
        \end{definition}

We can now give formulas for the numbers $\B(\pi_i)$ and $\F(\pi_i)$.

        \begin{lemma}
        \label{lem:B}
        Let $\alpha\in\PF_{m,n}$ with outcome permutation $\pi=\pi_1\pi_2\cdots\pi_n\in \Sym_{m,n}$. 
        For each $i\in[n]$ with $\pi_i>0$, if car $\pi_i$ found its preferred spot occupied and backed into its final parking spot at position $i$ (as described in Case \ref{item:case 1}), then the number of preferences for car $\pi_i$ is given by $\B(\pi_i)=\min(\R(\pi_i),k)$. Moreover, whenever $\pi_i=0$, then $\B(0)=0$.
        \end{lemma}
        \begin{proof}
            By \Cref{def:LookRight},  $\R(\pi_i)$ counts the number of cars $\pi_{i+1},\pi_{i+2},\ldots,\pi_{i+x}$ which parked before car $\pi_i$ entered the street and did so to the immediate right (and contiguously) of the spot $i$ in which car $\pi_i$ ultimately parks.
            Namely, the cars $\pi_{i+1},\pi_{i+2},\ldots,\pi_{i+x}$ parked in spots $i+1,i+2,\ldots,i+x$.
            Since car $\pi_i$ could prefer any of the first $k$ of these occupied spots in order to then back into the $i$th spot, we must have that  $\B(\pi_i)=\min(\R(\pi_i),k)$.

            Whenever $\pi_i=0$, this is not car, so it cannot have any preferences, hence $\B(0)=0$.
        \end{proof}

        Now, it also could be that car $\pi_i$ satisfies Case (\ref{item:case 2}), where the car parked in spot $i$ by first checking $k$ spaces back from its preference and then going forward up to $\ell$ spaces to park in spot $i$. To account for these preferences, we consider the cars that parked immediately to the left of spot $i$ that parked before car $\pi_i$ entered the street.

        \begin{lemma}\label{lem:F}
            Let $\alpha\in\PF_{m,n}$ with outcome permutation $\pi=\pi_1\pi_2\cdots\pi_n\in \Sym_{m,n}$. 
        For each $i\in[n]$, if 
        car $\pi_i$ found its preferred spot occupied and moved forward into its final parking spot (as described in Case \ref{item:case 2}), then 
        the number of preferences for car $\pi_i$ is given by 
            \begin{align*}
                \F(\pi_i) =
                \begin{cases}
                0&\text{if $\pi_i=0$}\\
                0&\text{if $\Le(\pi_i)=0$}\\
                    \min(i-1,\ell) & \text{if $0<\Le(\pi_i)=i-1$}\\
                    \max(\min(\Le(\pi_i)-k,\ell),0) & \text{if $0<\Le(\pi_i)<i-1$}.\\
                \end{cases}
            \end{align*}
            
        \end{lemma}
        
        \begin{proof}
        By definition of $\F(\pi_i)$ if $\F(\pi_1)=0$, as there are no cars to the left of $\pi_1$, which agrees with the given formula as $\min(1-1,\ell)=0$. Moreover, 
        for any $1<i\leq n$ such that $\Le(\pi_i)=0$, there are no cars to the left, hence $\F(\pi_i)=0$.
        Lastly, if $\pi_i=0$ this is not a car and so it cannot have preferences, namely $\F(0)=0$.
        We now consider $1<i\leq n$ for which $\pi_i>0$ and 
            argue using three cases: 
            \begin{enumerate}
            \item all spots to the left of spot $i$ are occupied before car $\pi_i$ parks, i.e., $\Le(\pi_i)=i-1$ \label{item: all spots to the left}
            \item not all spots to the left of spot $i$ 
            are occupied, i.e., $\Le(\pi_i)< i-1$, and 
            \begin{enumerate}
            \item $k<\Le(\pi_i)$, or
        \label{item: not all spots to the left}
            \item  $k\geq\Le(\pi_i)$.\label{item: not all spots to the left 2}
        \end{enumerate}
        \end{enumerate}
        
            For (\ref{item: all spots to the left}): 
            In this case all of the spots $1,2,\ldots i-1$ are occupied, so regardless of the value of $k$, if car $\pi_i$ prefers any of those spots, backing up any amount it will always find occupied spots or reach the start of the street. So the only parameter affecting the preferences for car $i$ is $\ell$.
            If all spots to the left of spot $i$ are filled before car $\pi_i$ enters to park, then $\Le(\pi_i)=i-1$. 
            This means spots $1,2,\ldots,i-1$ are occupied by cars $\pi_1,\pi_2,\ldots,\pi_{i-1}$, where those cars arrived and parked before car $\pi_i$ entered the street.
            Since car $\pi_i$ can only move forward $\ell$ spaces, it can only have preferred up to $\ell$ of the spaces to the left of spot $i$, namely the spots $i-\ell,i+1-\ell\ldots,i-1$. However, we need to account for the back that $i-\ell$ may be negative, hence the number of possible spots that car $\pi_i$ can prefer to as to move forward and park in spot $i$ is limited to $\min(\Le(\pi_i),\ell)=\min(i-1,\ell)$.            

            For (\ref{item: not all spots to the left}): Assume that not all spots to the left of $i$ are occupied and $k<\Le(\pi_i)$.
            Let $x,x+1,\ldots,i-1$ be the sequence of occupied spots where spot $x-1$ is empty, and $x>1$. Then $\Le(\pi_i)=(i-1)-x+1=i-x$, and hence $x=i-\Le(\pi_i)$.
            Let $a_i$ be the preference for car $\pi_i$. 
            If $x\leq a_i\leq x+k-1$, then car $\pi_i$ would find its preference occupied and checking up to $k$ spots behind its preference car $\pi_i$ would find spot $x-1$ empty, and would then park there. 
            Contradicting that car $\pi_i$ parks in spot $i$. 
            Thus $a_i\geq x+k$ (and recall in this case $a_i\leq i$). That is, car $\pi_i$ can
            prefer spots $x+k,x+k+1,\ldots,i-1$ and hence 
            it can have $\Le(\pi_i)-k$ preferences to the left of where it parked.
            However, in order for car $\pi_i$ to ultimately  park in spot $i$, these preferences must also be no more than $\ell$ spots from spot $i$. 
            These possible preferences are then limited to the parking spots numbered $i-\ell,i-\ell+1,\ldots,i-1$. 
            These conditions require that the preference $a_i$ for car $\pi_i$ satisfy
            \begin{align}
            x+k\leq & a_i\leq i-1\mbox{ and}\\
            i-\ell\leq & a_i\leq i-1.
            \end{align}
        As both inequalities must hold, we have that 
        $\max(x+k,i-\ell)\leq a_{i}\leq i-1$.
        Hence, since $a_i$ satisfies this inequality, the number of preferences that $a_i$ can have is 
        \[i-1-(\max(x+k,i-\ell))+1=i-\max(x+k,i-\ell).\]
        Now note that 
        \begin{align}
            i-\max(x+k,i-\ell)&=i-(-\min(-x-k,-i+\ell)=i+\min(-x-k,-i+\ell)\\
            &=\min(i-x-k,\ell).\label{eq:need to substitute}
        \end{align}
        Substituting $x=i-\Le(\pi_i)$ into \Cref{eq:need to substitute} yields
        \begin{align*}
        i-\max(x+k,i-\ell)&=\min(i-i+\Le(\pi_i)-k,\ell)=\min(\Le(\pi_i)-k,\ell).
        \end{align*}
Hence, the number of preferences is $\F(\pi_i)=\min(\Le(\pi_i)-k,\ell)$. As $\Le(\pi_i)-k> 0$, then 
$\min(\Le(\pi_i)-k,\ell)=\max(\min(\Le(\pi_i)-k,\ell),0)$. Therefore 
$\F(\pi_i)=\max(\min(\Le(\pi_i)-k,\ell),0)$, as claimed.

For (\ref{item: not all spots to the left 2}):  Next, we consider the case where $k\geq \Le(\pi_i)$. In this case, car $\pi_i$ could not have preferred any of the $\Le(\pi_i)$ spaces to the left of spot $i$, because then it would have backed up past those $\Le(\pi_i)$ spaces and parked in an available spot, contradicting that car $\pi_i$ parks in spot $i$.
So, in this case, car $\pi_i$ can prefer $0$ spots to the left of spot $i$. Thus, since $k\geq \Le(\pi_i)$, $\min(\Le(\pi_i)-k,\ell)$ is now negative, when we need it to be zero. Hence, the number of  preferences for car $\pi_i$ must satisfy 
$\F(\pi_i)=\max(\min(\Le(\pi_i)-k,\ell),0)$, as~desired.
\end{proof}

These lemmas establish the following result.
        
\begin{corollary}
        \label{cor:prefs for car i}
            Fix $\pi\in \Sym_{m,n}$. Then for each $i\in[n]$ with $\pi_i>0$, the number of preferences for car $\pi_i$ is given by 
           $
    |\Pref(\pi_i)|=\B(\pi_i)+\F(\pi_i) + 1.
            $
    
        \end{corollary}
        \begin{proof}
            It follows from Lemmas~\ref{lem:B} and \ref{lem:F}, that, if $\Pref(\pi_i)$ denotes the set of possible preferences of an actual car $\pi_i$, namely $\pi_i>0$, then by Definition \ref{def:cases_for_prefs}, the number of preferences is given by $\F(\pi_i)+\B(\pi_i)+1$, where the 1 comes from a car being able to park in its preference.
        \end{proof}

Whenever $\pi_i=0$, we have shown that $\F(0)=\B(0)=0$, but for convenience we let $|\Pref(\pi_i=0)|=1$.
        We can now use Corollary \ref{cor:prefs for car i} to count the number of  $(k,\ell)$-pullback $(m,n)$-parking functions parking the cars in the order $\pi$.

        \begin{lemma}\label{lem:parking functions for perm pi}
            Let $\mathcal{O}^{-1}(\pi)$ denote the set of pullback parking functions with $m$ cars and $n$ spots, $m\leq n$, parking the cars in the order $\pi$. Fix $\pi=\pi_1\pi_2\cdots\pi_n\in \Sym_{m,n}$. For any nonnegative integers $k,\ell$ and positive integers $1\leq m\leq n$, we have that
            \begin{align*}
                |\mathcal{O}^{-1}(\pi)|=|\{\alpha\in\PF_{m,n}(k,\ell):\mathcal{O}(\alpha)=\pi\}|=\left(\prod_{i=1}^n [\B(\pi_i)+\F(\pi_i)+1]\right).
            \end{align*}
            
        \end{lemma}
        
        \begin{proof}
            This follows directly from Corollary \ref{cor:prefs for car i} and by taking the product over all $i\in[n]$, where again we remark that if $\pi_i=0$, then $|\Pref(\pi_i=0)|=1$.
        \end{proof}
        
        We  now give our first formula to count $(k,\ell)$-pullback $(m,n)$-parking functions.
        
        \begin{theorem}\label{thm:counting through perms}
            Fix any nonnegative integers $k,\ell$ and positive integers $1\leq m\leq n$. The number of $(k,\ell)$-pullback $(m,n)$-parking functions is given by 
            \begin{align*}
                |\PF_{m,n}(k,\ell)|=
                \sum_{\pi\in \Sym_{m,n}} |\mathcal{O}^{-1}(\pi)|=
                \sum_{\pi=\pi_1\pi_2\cdots\pi_n\in \Sym_{m,n}}\left(\prod_{i=1}^n[\B(\pi_i)+\F(\pi_i)+1]\right).
            \end{align*}
            
        \end{theorem}

        \begin{proof}
            This follows directly from Lemma \ref{lem:parking functions for perm pi} and by taking the sum over all $\Sym_{m,n}$.
        \end{proof}

\subsection{Applications of main theorem}
        We conclude this section by specializing parameters in Theorem \ref{thm:counting through perms} which give new formulas for $k$-Naples $(m,n)$-parking functions and vacillating $(m,n)$-parking functions. We note that neither of these formulas have appeared in the literature. 
        
        We can now specify $\ell=n-1$, in which case $\PF_{m,n}(k,n-1)=\PF_{m,n}(k)$, as the parameter $\ell=n-1$ does not restrict the forward motion of cars. With this in mind, we arrive at the following formula for the number of $k$-Naples $(m,n)$-parking functions.
        
        \begin{corollary}\label{cor:counting through perms}
Let $\widetilde{\B}$ and $\widetilde{\F}$ denote the specialization of $\F$ and $\B$ (from Definition \ref{def:cases_for_prefs}) with $n=\ell-1$, respectively.  Then
            the number of $k$-Naples parking functions with $m$ cars and $n$ spots such that $1\leq m\leq n$ is given by
            \[
            |\PF_{m,n}(k)|=\sum_{\pi=\pi_1\pi_2\cdots\pi_n\in \Sym_{m,n}}\left(\prod_{i=1}^n[\widetilde{\B}(\pi_i)+\widetilde{\F}(\pi_i)+1]\right).
            \]
        \end{corollary}
        
        Fang et al.~ in \cite{fang2024vacillatingparkingfunctions} defined vacillating parking functions (for $m=n$), which are preference lists in which cars check their preference, 
        the spot behind their preference (if it exists), and the spot ahead of their preference (if it exists) in this order and park in the first available spot they encounter among those options (if possible). 
        If any car is unable to park among those parking spots, then the list of preferences is not a vacillating parking function. 
        Let $\VPF_n$ denote the set of vacillating parking functions.
        A recursive formula for the number of vacillating parking functions was given in 
        \cite[Theorem 2.1]{fang2024vacillatingparkingfunctions}. 
        We extend this definition and let $\VPF_{m,n}$ denote the set of vacillating parking functions  with $m$ cars and $n$ spots such that $1\leq m\leq n$. 
        Setting $k=\ell=1$ in the formula of Theorem \ref{thm:counting through perms}, we provide a generalization to their result in the case where there are $m$ cars and $n$ spots, thereby establishing the following.
        
        \begin{corollary}
        Let $\widehat{\B}$ and $\widehat{\F}$ denote the specialization of $\F$ and $\B$ (from Definition \ref{def:cases_for_prefs}) with $k=\ell=1$, respectively.  Then
            the number of vacillating parking functions with $m$ cars and $n$ spots such that $1\leq m\leq n$ is given by
            \begin{align*}    
            |\VPF_{m,n}|=
            \sum_{\pi=\pi_1\pi_2\cdots\pi_n \in \Sym_{m,n}}\left(\prod_{i=1}^n[\widehat{\B}(\pi_i)+\widehat{\F}(\pi_i)+1]\right).
            \end{align*}
        \end{corollary}
              
    \section{Recursive formula}\label{sec:recursive formula}
        In this section, we provide a recursive formula for the number of $(k,\ell)$-pullback $(m,n)$-parking functions. 
        To make our approach precise we begin by adapting the definition of ``contained parking functions'' as first given in \cite{knaple}.

        \begin{definition}\label{def:contained}
         Let $a\leq b$, and consider a street with a spot $0$ added before spot $1$. 
            The set of \emph{contained $(k,\ell)$-pullback $(a,b)$-parking functions}, denoted $\C_{a,b}(k,\ell)$, is the subset of $(k,\ell)$-pullback $(a,b)$-parking functions such that the $a$ cars all park between spots 1 through $b$, and none of the cars back into spot $0$ while attempting to park.
        \end{definition}

        In Definition \ref{def:contained}, we use the word ``contained'', because if one were to introduce one or more available spots to the ends of the parking lot (before the first spot and after the $b$th spot), the $a$ cars would be ``contained'' only in spots $1,2,\ldots,b$. We also use $a$ and $b$ here instead of $m$ and $n$ because our contained pullback parking functions will focus only on a subset of size $a$ of the $a\leq m$ parked cars on a portion of the street of length $b$ when there are $b\leq n$ spots.
        
We illustrate the need for the set of contained parking functions via the following example.
\begin{example}
Consider the street of length $n=12$, with $m=7$ cars. Let $k=3$ and $\ell=2$, and consider the outcome $\pi=007103002654\in \mathfrak{S}_{7,12}$.

We want to count the number of $(3,2)$-pullback $(7,12)$-parking functions with outcome $\pi$. In order to do this, we consider the non-zero entries in $\pi$. For example, we need to consider how many pullback parking functions will have cars $2,6,5,4$ parked in spots $9,10,11,12$ on the street. 

This will not be the same as counting all parking functions in the set $\PF_{4,4}(3,2)$ and then simply adjusting the preferences. Consider the parking preference list $(1,4,4,1)$, which is a pullback parking function whose outcome is $\sigma = 1432$. 
However, when we adjust the preferences to be $(9,12,12,9)$ for cars $2,4,5,6$, we would arrive at the outcome $62054$ in spots $8$ through $12$. 

This means that we need to consider a subset of the parking functions in $\PF_{4,4}(3,2)$, and not necessarily the full set in order to solve our problem. 
\end{example}

Our main result in enumerating $(k,\ell)$-pullback $(a,b)$-parking functions will depend on the number of contained $(k,\ell)$-pullback $(m,n)$-parking functions. Thus we first give a formula for $|\C_{a,b}(k,\ell)|$. 

\subsection{Contained parking functions}\label{subsec:contained}
The goal is to count pullback parking functions with more spots than cars. However, one of the difficulties in counting these recursively is that we need to know the lengths
of the longest sub-intervals (consisting of adjacent parking spots) on the street which contain the cars. 
To this end we set the following notation.

\begin{notation}\label{special notation}
Given a permutation $\pi=\pi_1\pi_2\cdots\pi_n\in\Sym_{m,n}$, we let 
$S$ be the subset of $[n]$ of size $m$ in which the cars parked. Namely,
$S=\{u\in[n]:\pi_u>0\}$.
We then partition $S$ into maximal subintervals consisting of consecutive entries and we denote these subintervals by $S_1,S_2,\ldots, S_j$. Moreover, for each $1\leq i\leq j$, let $t_i$  be the cardinality of the set $S_i$. Then, for each $1\leq i\leq j$, we 
define $T_i\subseteq[m]$ of size $t_i$, to be the set of cars parking in subinterval $S_i$. Namely, for each $1\leq i\leq j$, we let
$T_i=\{\pi_u:u\in S_i\}$.
\end{notation}

\begin{example}
 Let $n=10, m=6$ and $\pi=0236001540\in \Sym_{6,10}$. Then $S=\{2,3,4,7,8,9\}$ and its partition into maximal subintervals consisting of consecutive entries are $S_1=\{2,3,4\}$, and $S_2=\{7,8,9\}$. Now $t_1=3, t_2=3$ with $T_1=\{2,3,6\}$ and $T_2=\{1,4,5\}$.
\end{example}

In what follows we will count the preferences of the cars in $T_i$ so that they park in the subinterval $S_i$. These preferences will form a contained parking function in which the number of spots and cars are both equal to the length of the subinterval. 
We will show that, for a subinterval $S_i$, the number of such preferences is equal to $|\C_{t_i,t_i}(k,\ell)|$.

To make this approach precise we  introduce parking outcomes on a subinterval of the street. We do this next.
            
\begin{definition}
Fix a subset $T\subseteq[m]$ of size $t$, and let $\Sym_T$ denote the set of permutations $\pi=\pi_1 \pi_2\cdots \pi_t$ of the set $T$. 
Define the following set of permutations 
\[\mathfrak{T}_{T}=\{\pi_0\pi_1\cdots\pi_t:\pi_0=0\mbox{ and } \pi_1\pi_2\cdots\pi_t\in \Sym_T \},\]   
which consists of all permutations in $\Sym_T$ to which we append a zero at the start.
\end{definition}
We remark that the placement of a zero at the start of the permutations $0\pi_1\pi_2\cdots\pi_t\in \Tym_{T}$, will ensure that from the set of contained pullback parking functions on a subinterval we can construct all pullback parking functions with outcome $\pi_1\pi_2\cdots\pi_n\in \Sym_{T}$.

In what follows, we specialize the results in \Cref{sec:counting through permutations} to contained pullback parking functions on a subinterval.
            
\begin{definition}\label{def:lots of stuff}
Following \Cref{special notation}, we let 
$S$ be an sub-interval of $[n]$ with length $t$, and $T$ be the corresponding set of cars parked in the sub-interval $S$.
Fix a permutation $\pi=\pi_0\pi_1\pi_2\cdots\pi_t\in\Tym_T$ and, for each $1\leq i\leq t$, let 
\begin{itemize}
    \item $\R(\pi_i)$ be the longest subsequence $\pi_{i+1},\pi_{i+2},\ldots,\pi_{i+x}$ where $\pi_{y}<\pi_i$ for all $i+1\leq y\leq i+x\leq t$; and 
    \item  $\Le(\pi_i)$ be the length of the longest consecutive subsequence, $\pi_y,\pi_{y+1},\ldots,\pi_{i-1}$, such that $0<\pi_z<\pi_i$ for all $1\leq y\leq z\leq i-1$.

\end{itemize}
Then, as in \Cref{def:cases_for_prefs}, let 
\begin{itemize}
\item $\B(\pi_i)$ be the number of preferences for car $\pi_i$ satisfying Case (\ref{item:case 1}), and
\item $\F(\pi_i)$ be the number of preferences for car $\pi_i$ satisfying Case (\ref{item:case 2}).
\end{itemize}   
\end{definition}

\begin{corollary}\label{cor: B contained}
If $S$, $T$, $t$, and $\pi\in\Tym_T$ are as in \Cref{def:lots of stuff}, then, for each $1\leq v\leq t$, we have that
$\B(\pi_v)=\min(\R(\pi_v),k)$, and 
        $\F(\pi_v)=\max(\min(\Le(\pi_i)-k),\ell),0)$.
\end{corollary}
\begin{proof}
On each subinterval $S$ of length $t$, we are parking the cars in $T$, which consists of $t$ cars. Thus, result is a special case of \Cref{lem:B} and \Cref{lem:F}, where the number of spots and cars are equal. 
\end{proof}

\begin{corollary} \label{cor:prefs for car i contained}
If $S$, $T$, $t$, and $\pi\in\Tym_T$ are as in \Cref{def:lots of stuff}. For each $1\leq v\leq t$, let $\Pref(\pi_v)$ denote the set of preferences of car $\pi_v$ which cause car $\pi_v$ to park in spot $v$. Then 
\begin{align*}
    |\Pref(\pi_v)|=\B(\pi_v)+\F(\pi_v) + 1.
\end{align*}
\end{corollary}
\begin{proof}
This follows from \Cref{cor: B contained} and the definitions of  $\B(\pi_i)$ and $\F(\pi_i)$. The addition of $1$ accounts for car $i$ preferring and parking in spot $i$. 
\end{proof}

In what follows we will use the previous results to count pullback parking functions. \begin{remark}\label{remark on complicated notation}We remark that the process for counting these parking functions is as follows:
\begin{enumerate}
    \item Select a subset $S\subseteq[n]$ of size $m$, consisting of the indices for the spots in which the cars park.
    \item Partition $S$ into maximal subintervals consisting of consecutive entries and denote these subintervals by $S_1,S_2,\ldots, S_j$. 

    \item Let $t_1,t_2,\ldots,t_j$ denote the length of the subintervals $S_1,S_2,\ldots,S_j$, respectively.
    \item Partition the set $[m]$ into subsets $T_1,T_2,\ldots,T_j$ of sizes $t_1,t_2,\ldots,t_j$.
    \item For each $1\leq v\leq j$, the cars in $T_v$ will park in the subinterval $S_v$.
\item For each $1\leq v\leq j$, the set of permutations in $\Tym_{T_v}$ gives all possible parking outcomes of the cars in $T_v$ parking in the subinterval $S_v$.
\end{enumerate}
\end{remark}
Next we consider a single subinterval and count the number of contained parking functions whose outcome is a permutation $\pi\in\Tym_{T_v}$.
            
\begin{lemma}\label{lem:contained subinterval with a fixed outcome}
    Fix $1\leq v\leq j$, and let $S_v$, $T_v$, $t_v$, and $\pi=\pi_0\pi_1\pi_2\cdots\pi_{t_v}\in\Tym_{T_v}$ be as in 
\Cref{remark on complicated notation}. Let $|\mathcal{O}^{-1}(\pi)|$ denote the number of contained $(k,\ell)$-pullback $(t_v,t_v)$-parking functions whose outcome is the permutation $\pi$. Then        
    \begin{align*}
        |\mathcal{O}^{-1}(\pi)|=\left(\prod_{i=1}^{t_v} [\B(\pi_i)+\F(\pi_i)+1]\right).
    \end{align*}
\end{lemma}
            
\begin{proof}
This follows directly from Corollary \ref{cor:prefs for car i contained} and taking the product over all $1\leq i\leq t_v$.
\end{proof}

Next we consider a single subinterval and count the number of contained parking functions over all possible outcomes i.e., over all permutations $\pi\in\Tym_{T_v}$.

\begin{theorem}\label{thm:pb contained}
Fix $1\leq v\leq j$, and let $S_v$, $T_v$, $t_v$, and $\pi=\pi_0\pi_1\pi_2\cdots\pi_{t_v}\in\Tym_{T_v}$ be as in 
\Cref{remark on complicated notation}. Let $|\mathcal{O}^{-1}(\pi)|$ denote the number of contained $(k,\ell)$-pullback $(t_v,t_v)$-parking functions whose outcome is the permutation $\pi$. Then the number of contained $(k,\ell)$-pullback $(t_v,t_v)$-parking functions is given by 
\begin{align*}
    |\C_{t_v,t_v}(k,\ell)|=\sum_{\pi\in \Tym_{T_v}} |\mathcal{O}^{-1}(\pi)|=\sum_{\pi\in \Tym_{T_v}}\left(\prod_{i=1}^{t_v}[\B(\pi_i)+\F(\pi_i)+1]\right).
\end{align*}
\end{theorem}
\begin{proof}
    This follows from \Cref{lem:contained subinterval with a fixed outcome} by taking a sum over all possible parking outcomes in the set $\Tym_{T_v}$.
\end{proof}

In the following result, the specialization of $\ell=n-1$, means the result holds for $k$-Naples $(m,n)$-parking functions.
\begin{corollary}\label{coro:kNaples_contained_count}
    Fix $1\leq v\leq j$, and let $S_v$, $T_v$, $t_v$, and $\pi=\pi_0\pi_1\pi_2\cdots\pi_{t_v}\in\Tym_{T_v}$ be as in 
\Cref{remark on complicated notation}.
                Let $|\mathcal{O}^{-1}(\pi)|$ denote the number of 
                contained $(k,n-1)$-pullback $(t_v,t_v)$-parking functions whose outcome is the permutation $\pi$. That is, the number of contained $k$-Naples $(t_v,t_v)$-parking functions. 
                Then the number of contained contained $k$-Naples $(t_v,t_v)$-parking functions is given by 
                \begin{align*}
                    |\C_{t_v,t_v}(k,n-1)|=
                    \sum_{\pi\in \Tym_{T_v}} |\mathcal{O}^{-1}(\pi)|=
                    \sum_{\pi\in \Tym_{T_v}}\left(\prod_{i=1}^{t_v}[\B(\pi_i)+\F(\pi_i)+1]\right).
                \end{align*}
\end{corollary}

\subsection{Recursions} 

In this section we give a recursion for $(k,\ell)$-pullback $(m,n)$-parking functions. Throughout whenever $a>b$, $|\PF_{a,b}|=0$ as there are more cars than spots.
We begin with the following definition which  plays a key role in our proofs.
        \begin{definition}\label{def:popular region}
        
            The \emph{popular region} is the longest contiguous set of cars parked immediately to one side of spot $i$. We let $R$ denote  the number of cars parked in the popular region. The specific side of spot $i$  which contains the popular region is noted whenever it is used.
        \end{definition}
        
        We  now give a recursive formula for the number of pullback $(m,n)$-parking functions.
        
        \begin{theorem}\label{thm:pullback}
        
            Let $m$ and $n$ be positive integers with $m\leq n$, and fix $0\leq k \leq n-1$ and $0\leq \ell \leq n-1$. The number of $(k,\ell)$-pullback $(m,n)$-parking functions satisfies the following recursive formula 
            \begin{align*}
                |\PF_{m,n}(k,\ell)|=
                &\sum_{i=1}^n\left[X(i)+
                Y(i)+
                \sum_{x=0}^{m-1}\left(Z(i,x)+\sum_{R=1}^{n-i-1}V(i,x,R)+
                \sum_{R=k+1}^{i-2}W(i,x,R)\right)\right],
            \end{align*}
            where\\
                $X(i)=\binom{m-1}{n-i}
                |\PF_{m-1-n+i, i-1}(k,\ell)||\C_{n-i,n-i}(k,\ell)| \min(k,n-i)$,\\
                $Y(i)= \binom{m-1}{i-1}|\PF_{i-1,i-1}(k,\ell)| |\C_{m-i,n-i}(k,\ell)|  \min(i-1,\ell)$,\\
                $Z(i,x)= \binom{m-1}{x}|\PF_{x,i-1}(k,\ell)|  |\C_{m-1-x,n-i}(k,\ell)|$,\\
            $V(i,x,R)=\binom{m-1}{x}|\PF_{x,i-1}(k,\ell)|\binom{m-1-x}{R}|\C_{R,R}(k,\ell)||\C_{m-1-x-R,n-R-i-1}(k,\ell)| \min(R,k)$, and $
                W(i,x,R)=\binom{m-1}{x} |\PF_{x,i-R-2}(k,\ell)|\binom{m-1-x}{R}|\C_{R,R}(k,\ell)| |\C_{m-1-x-R,n-i}(k,\ell)| \min(R-k,\ell)$.
        \end{theorem}
        
        \begin{proof}
        Let $\alpha = (a_1, a_2, \ldots , a_m)\in \PF_{m,n}(k,\ell)$.
            We count the number of $(k,\ell)$-pullback $(m,n)$-parking functions by partitioning the set based on the parking location of the final car relative to its preference. This gives rise to the following three cases:
            
            \begin{enumerate}[leftmargin=.7in]
                
                \item[\textbf{Case 1:}] Car $m$ prefers spot $i$ and parks there ($a_m = i$),
                
                \item[\textbf{Case 2:}] Car $m$ prefers some spot to the right of spot $i$ and backs into spot $i$ ($a_m > i$), or 
                
                \item[\textbf{Case 3:}] Car $m$ prefers a spot to the left of spot $i$ and it pulls forward into spot $i$ ($a_m < i$).
                
            \end{enumerate}
            
            We now count the preferences in each of these cases independently. 
            \smallskip
            
            \noindent{\textbf{Case 1:}}
            In this case, car $m$ prefers spot $i$ and parks there. We illustrate this case in Figure~\ref{fig:case1}.
            
            \begin{figure}[ht]
\centering\tikzset{every picture/.style={line width=0.75pt}} 
\begin{tikzpicture}[x=0.75pt,y=0.75pt,yscale=-1,xscale=1]

\draw [fill={rgb, 255:red, 0; green, 0; blue, 0 }  ,fill opacity=1 ][line width=2.25]    (84,80.6) -- (510,80.6) ;
\draw    (84,71) -- (84,90.6) ;
\draw    (510,71) -- (510,90.6) ;
\draw  [fill={rgb, 255:red, 255; green, 255; blue, 255 }  ,fill opacity=1 ] (297,80.6) .. controls (297,77.95) and (299.15,75.8) .. (301.8,75.8) .. controls (304.45,75.8) and (306.6,77.95) .. (306.6,80.6) .. controls (306.6,83.25) and (304.45,85.4) .. (301.8,85.4) .. controls (299.15,85.4) and (297,83.25) .. (297,80.6) -- cycle ;
\draw   (296,81) .. controls (296,76.33) and (293.67,74) .. (289,74) -- (197.82,74) .. controls (191.15,74) and (187.82,71.67) .. (187.82,67) .. controls (187.82,71.67) and (184.49,74) .. (177.82,74)(180.82,74) -- (92,74) .. controls (87.33,74) and (85,76.33) .. (85,81) ;
\draw   (510,79) .. controls (510,74.33) and (507.67,72) .. (503,72) -- (416.45,72) .. controls (409.78,72) and (406.45,69.67) .. (406.45,65) .. controls (406.45,69.67) and (403.12,72) .. (396.45,72)(399.45,72) -- (315,72) .. controls (310.33,72) and (308,74.33) .. (308,79) ;
\draw   (85,82) .. controls (85,86.67) and (87.33,89) .. (92,89) -- (179,89) .. controls (185.67,89) and (189,91.33) .. (189,96) .. controls (189,91.33) and (192.33,89) .. (199,89)(196,89) -- (289,89) .. controls (293.67,89) and (296,86.67) .. (296,82) ;
\draw   (308,82) .. controls (308,86.67) and (310.33,89) .. (315,89) -- (396.49,89) .. controls (403.16,89) and (406.49,91.33) .. (406.49,96) .. controls (406.49,91.33) and (409.82,89) .. (416.49,89)(413.49,89) -- (503,89) .. controls (507.67,89) and (510,86.67) .. (510,82) ;

\draw (196,47) node [anchor=north][inner sep=0.75pt]   [align=left] {{$i-1$ \quad \fontfamily{ptm}\selectfont spaces}};
\draw (411,47) node [anchor=north][inner sep=0.75pt]   [align=left] {{$n-i$ \quad \fontfamily{ptm}\selectfont spaces}};
\draw (189,98) node [anchor=north][inner sep=0.75pt]   [align=left] {{$x$ \quad \fontfamily{ptm}\selectfont cars}};
\draw (430,98) node [anchor=north][inner sep=0.75pt]   [align=left] {{$m-1-x$ \quad \fontfamily{ptm}\selectfont cars \quad}};
\draw (80,96.4) node [anchor=north west][inner sep=0.75pt]    {$1$};
\draw (299,96.4) node [anchor=north west][inner sep=0.75pt]    {$i$};
\draw (504,96.4) node [anchor=north west][inner sep=0.75pt]    {$n$};
\end{tikzpicture}
\vspace{-.35in}
                \caption{
                In the figure, spot $i$ is left open and the region to the left of spot $i$ consists of spots $1$ through $i-1$, and there are $x$ cars parked in those spots.
                Hence, $x\leq i-1$ and $x \leq m-1$. 
                The region to the right of spot $i$ consists of the spots $i+1$ through $n$, and $m-1-x$ cars are parked in that region. Hence, $m-1-x \leq n-i$.}
                \label{fig:case1}
                
            \end{figure}
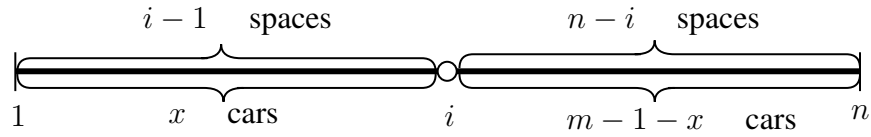
            
            Suppose that $x$ cars park to the left of spot $i$ such that $0\leq x \leq m-1$. 
            There are $\binom{m-1}{x}$ way to select those cars among $m-1$ cars. 
            We then park those cars in the $i-1$ spots to the left of spot $i$, which can be done $|\PF_{x,i-1}(k,\ell)|$ ways. Note that this returns $0$ if $x>i-1$.  
            
            We now consider the cars parked to the right of spot $i$. The cars parking in this region are uniquely determined, as the complement of the set of cars parking in the spots $1$ through $i-1$. Because we are parking these cars to the right of spot $i$, in order to leave spot $i$ vacant, we must only include parking functions which do not back into spot $i$ using the pullback parking rule. 
            This is precisely the set of contained pullback parking functions, $\C_{m-1-x,n-i}(k,\ell)$.  
            Hence, we have a factor of $|\C_{m-1-x,n-i}(k,\ell)|$ and this value is zero whenever $m-1-x > n-i$.
            
            The only car left to park is car $m$. 
            Given the assumption of this case, car $m$ prefers spot $i$, and parks there, accounting for a factor of $1$ in the count.
            To conclude we must sum over all possible values of $x$ and~$i$. Hence, the total count contributed by Case 1 is given by
    \begin{align}\label{case1}
        \sum_{i=1}^n\sum_{x=0}^{m-1}Z(i,x)\coloneqq\sum_{i=1}^n\sum_{x=0}^{m-1} \binom{m-1}{x} \cdot |\PF_{x,i-1}(k,\ell)| \cdot |\C_{m-1-x,n-i}(k,\ell)|
    .
            \end{align}
            
            \noindent{\textbf{Case 2:}}
            In this case, car $m$ prefers some spot to the right of spot $i$ and backs into spot $i$.
            
            In this case we set the popular region, as defined in Definition \ref{def:popular region}, to be located to the right of spot $i$. Recall that we let $R$ denote the length of the popular region.
            
            We now consider the following two subcases:
            \begin{enumerate}[leftmargin=1in]
                
                \item[\textbf{Subcase 2a:}] The popular region extends all the way to the end of the street ($R=n-i$). 
                
                \item[\textbf{Subcase 2b:}] The popular region does not extend to the end of the street ($R<n-i$).
                
            \end{enumerate}
            
            We now count the preferences for the cars in each of these subcases.
            
            \noindent \textbf{Subcase 2a:} The popular region extends all the way to the end of the street, so $R=n-i$. We illustrate this case in Figure \ref{fig:case2a}.
            
            \begin{figure}[ht]
            
                \centering\tikzset{every picture/.style={line width=0.75pt}} 
\begin{tikzpicture}[x=0.75pt,y=0.75pt,yscale=-1,xscale=1]

\draw [fill={rgb, 255:red, 0; green, 0; blue, 0 }  ,fill opacity=1 ][line width=2.25]    (104,100.6) -- (530,100.6) ;
\draw    (104,91) -- (104,110.6) ;
\draw    (530,91) -- (530,110.6) ;
\draw  [fill={rgb, 255:red, 255; green, 255; blue, 255 }  ,fill opacity=1 ] (317,100.6) .. controls (317,97.95) and (319.15,95.8) .. (321.8,95.8) .. controls (324.45,95.8) and (326.6,97.95) .. (326.6,100.6) .. controls (326.6,103.25) and (324.45,105.4) .. (321.8,105.4) .. controls (319.15,105.4) and (317,103.25) .. (317,100.6) -- cycle ;
\draw   (316,100) .. controls (316,95.33) and (313.67,93) .. (309,93) -- (217.82,93) .. controls (211.15,93) and (207.82,90.67) .. (207.82,86) .. controls (207.82,90.67) and (204.49,93) .. (197.82,93)(200.82,93) -- (112,93) .. controls (107.33,93) and (105,95.33) .. (105,100) ;
\draw   (529,99) .. controls (529,94.33) and (526.67,92) .. (522,92) -- (435.97,92) .. controls (429.3,92) and (425.97,89.67) .. (425.97,85) .. controls (425.97,89.67) and (422.64,92) .. (415.97,92)(418.97,92) -- (335,92) .. controls (330.33,92) and (328,94.33) .. (328,99) ;
\draw   (105,102) .. controls (105,106.67) and (107.33,109) .. (112,109) -- (199,109) .. controls (205.67,109) and (209,111.33) .. (209,116) .. controls (209,111.33) and (212.33,109) .. (219,109)(216,109) -- (309,109) .. controls (313.67,109) and (316,106.67) .. (316,102) ;
\draw   (328,102) .. controls (328,106.67) and (330.33,109) .. (335,109) -- (416,109) .. controls (422.67,109) and (426,111.33) .. (426,116) .. controls (426,111.33) and (429.33,109) .. (436,109)(433,109) -- (522,109) .. controls (526.67,109) and (529,106.67) .. (529,102) ;
\draw [color={rgb, 255:red, 208; green, 2; blue, 27 }  ,draw opacity=1 ][line width=4.5]    (327,100.6) -- (530,100.6) ;

\draw (171,66.4) node [anchor=north west][inner sep=0.75pt]    {$i-1$ \quad \fontfamily{ptm}\selectfont spaces};
\draw (361,66.4) node [anchor=north west][inner sep=0.75pt]    {$R=n-i$ \quad \fontfamily{ptm}\selectfont spaces};
\draw (134,116.4) node [anchor=north west][inner sep=0.75pt]    {$m-1-( n-i)$ \quad \fontfamily{ptm}\selectfont cars};
\draw (367,116.4) node [anchor=north west][inner sep=0.75pt]    {$R=n-i$ \quad \fontfamily{ptm}\selectfont cars};
\draw (99,115.4) node [anchor=north west][inner sep=0.75pt]    {$1$};
\draw (319,116.4) node [anchor=north west][inner sep=0.75pt]    {$i$};
\draw (525,116.4) node [anchor=north west][inner sep=0.75pt]    {$n$};
\end{tikzpicture}
                \vspace{-.35in}
                \caption{In the figure, spot $i$ is vacant. 
                The region to the right of spot $i$ consists of the spots numbered $i+1$ through $n$. In this case, the popular region  is highlighted in red. Here, the popular region satisfies $R=n-i$. 
                The region to the left of spot $i$ consists of the spots numbered $1$ through $i-1$, and the remaining $m-1-(n-i)$ cars are parked there.}
                
                \label{fig:case2a}
                
            \end{figure}
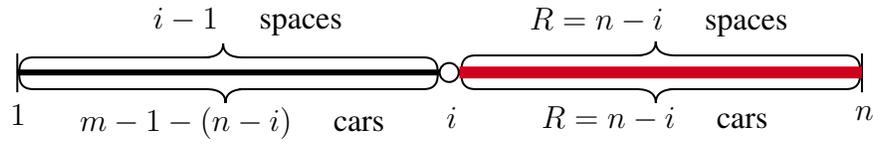
            
            Given the constraints in Case 2a, we first choose the $n-i$ cars to park to the right of spot $i$. This can be done in $\binom{m-1}{n-i}$ ways.
            The complementary set consisting of $m-1-(n-i)$ cars are then parked to the left of spot $i$, which can park in
            $|\PF_{m-1-n+i,i-1}(k,\ell)|$ ways.
            
            For the cars parking in spots $i+1$ through $n$, we park the $R=n-i$ cars using the contained pullback parking function in order to not back up into spot $i$. This can be done in $|\C_{R,R}(k,\ell)|=|\C_{n-i,n-i}(k,\ell)|$ many ways.
            
            Now, we consider the number of preferences that car $m$ can have, which allow it to park in spot $i$. 
            In some cases, this will be $k$, as car $m$ can prefer at most $k$ spots ahead of spot $i$ and still be able to back into spot $i$; however, we must account for the case when $i+k>n$. In that case, the number of possible preferences for car $m$ is simply all available spots in the popular region, namely $R=n-i$. Thus, the number of preferences for car $m$ is counted by $\min(k,n-i)$.
            
            To conclude, we must sum over all possible values of $i$ such that $1\leq i\leq n$.
            Hence, the total count contributed by Subcase~2a is given by
\begin{align}\label{case2a}
                \sum_{i=1}^n X(i)\coloneqq\sum_{i=1}^n \binom{m-1}{n-i}|\PF_{m-1-n+i, i-1}(k,\ell)|\cdot|\C_{n-i,n-i}(k,\ell)|\cdot \min(k,n-i).
            \end{align}
            
            \noindent \textbf{Subcase 2b:} Not all of the spaces to the right of spot $i$ are occupied, so the popular region does not extend to the end of the street ($R<n-i$). We illustrate this case in Figure \ref{fig:case2b}.
            
            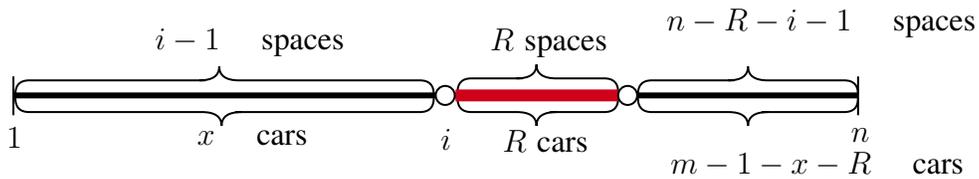
\begin{figure}[ht]
            
                \centering\tikzset{every picture/.style={line width=0.75pt}} 

\begin{tikzpicture}[x=0.75pt,y=0.75pt,yscale=-1,xscale=1]

\draw [fill={rgb, 255:red, 0; green, 0; blue, 0 }  ,fill opacity=1 ][line width=2.25]    (104,100.6) -- (530,100.6) ;
\draw    (104,91) -- (104,110.6) ;
\draw    (530,91) -- (530,110.6) ;
\draw  [fill={rgb, 255:red, 255; green, 255; blue, 255 }  ,fill opacity=1 ] (317,100.6) .. controls (317,97.95) and (319.15,95.8) .. (321.8,95.8) .. controls (324.45,95.8) and (326.6,97.95) .. (326.6,100.6) .. controls (326.6,103.25) and (324.45,105.4) .. (321.8,105.4) .. controls (319.15,105.4) and (317,103.25) .. (317,100.6) -- cycle ;
\draw   (316,100) .. controls (316,95.33) and (313.67,93) .. (309,93) -- (217.82,93) .. controls (211.15,93) and (207.82,90.67) .. (207.82,86) .. controls (207.82,90.67) and (204.49,93) .. (197.82,93)(200.82,93) -- (112,93) .. controls (107.33,93) and (105,95.33) .. (105,100) ;
\draw   (409,99) .. controls (409,94.33) and (406.67,92) .. (402,92) -- (377.79,92) .. controls (371.12,92) and (367.79,89.67) .. (367.79,85) .. controls (367.79,89.67) and (364.46,92) .. (357.79,92)(360.79,92) -- (335,92) .. controls (330.33,92) and (328,94.33) .. (328,99) ;
\draw   (105,102) .. controls (105,106.67) and (107.33,109) .. (112,109) -- (199,109) .. controls (205.67,109) and (209,111.33) .. (209,116) .. controls (209,111.33) and (212.33,109) .. (219,109)(216,109) -- (309,109) .. controls (313.67,109) and (316,106.67) .. (316,102) ;
\draw   (328,102) .. controls (328,106.67) and (330.33,109) .. (335,109) -- (357.8,109) .. controls (364.47,109) and (367.8,111.33) .. (367.8,116) .. controls (367.8,111.33) and (371.13,109) .. (377.8,109)(374.8,109) -- (402,109) .. controls (406.67,109) and (409,106.67) .. (409,102) ;
\draw [color={rgb, 255:red, 208; green, 2; blue, 27 }  ,draw opacity=1 ][line width=4.5]    (327,100.6) -- (409,100.6) ;
\draw   (419,102) .. controls (419,106.67) and (421.33,109) .. (426,109) -- (464.5,109) .. controls (471.17,109) and (474.5,111.33) .. (474.5,116) .. controls (474.5,111.33) and (477.83,109) .. (484.5,109)(481.5,109) -- (523,109) .. controls (527.67,109) and (530,106.67) .. (530,102) ;
\draw  [fill={rgb, 255:red, 255; green, 255; blue, 255 }  ,fill opacity=1 ] (409,100.6) .. controls (409,97.95) and (411.15,95.8) .. (413.8,95.8) .. controls (416.45,95.8) and (418.6,97.95) .. (418.6,100.6) .. controls (418.6,103.25) and (416.45,105.4) .. (413.8,105.4) .. controls (411.15,105.4) and (409,103.25) .. (409,100.6) -- cycle ;
\draw   (530,100) .. controls (530,95.33) and (527.67,93) .. (523,93) -- (484.5,93) .. controls (477.83,93) and (474.5,90.67) .. (474.5,86) .. controls (474.5,90.67) and (471.17,93) .. (464.5,93)(467.5,93) -- (426,93) .. controls (421.33,93) and (419,95.33) .. (419,100) ;

\draw (174,65.4) node [anchor=north west][inner sep=0.75pt]    {$i-1$ \quad \fontfamily{ptm}\selectfont spaces \quad};
\draw (343,65.4) node [anchor=north west][inner sep=0.75pt]    {$R$ \fontfamily{ptm}\selectfont spaces \quad};
\draw (432,55.4) node [anchor=north west][inner sep=0.75pt]    {$n-R-i-1$ \quad \fontfamily{ptm}\selectfont spaces \quad};
\draw (194,115.4) node [anchor=north west][inner sep=1.75pt]    {$x$ \quad \fontfamily{ptm}\selectfont cars \quad};
\draw (348,115.4) node [anchor=north west][inner sep=1.75pt]    {$R$ \fontfamily{ptm}\selectfont cars \quad};
\draw (421,114.5) node [anchor=north west][inner sep=10.5pt]    {$m-1-x-R$ \quad \fontfamily{ptm}\selectfont cars\quad};
\draw (99,115.4) node [anchor=north west][inner sep=0.75pt]    {$1$};
\draw (318,116) node [anchor=north west][inner sep=0.75pt]    {$i$};
\draw (525,115.4) node [anchor=north west][inner sep=0.75pt]    {$n$};
\end{tikzpicture}
                \vspace{-.35in}
                \caption{In the figure, spot $i$ is left vacant. The region to the left of spot $i$ consists of the spots numbered $1$ through $i-1$, and there are  $x$ cars parked in those spots, with $x\leq i-1$. The region to the immediate right of spot $i$ consists of spots $i+1$ through $i+R$. The popular region is highlighted in red and contains exactly $R$ cars. 
                Moreover, the spot immediately to the right of the popular region is vacant. 
                The final region, to the far right of the street, consists of spots $i+R+2$ through $n$, and the the remaining $m-1-x-R$ cars park there such that $m-1-x-R\leq n-R-i-1$.}\label{fig:case2b}
                
            \end{figure}
            
            First, consider the section of the street to the left of spot $i$. 
            Let $x$ be the number of cars parked in this section of the street, where $0\leq x \leq m-1$. 
            There are $\binom{m-1}{x}$ ways to select the cars to park in those spots. 
            They are parked according to the pullback parking rule with $x$ cars and $i-1$ spaces, contributing the $|\PF_{x,i-1}(k,\ell)|$ to the count. 
            We then must sum over all possible values of $0\leq x\leq i-1$.
            
            Now consider all cars parked to the right of spot $i$. There are $n-i$ total spots in this region, and $m-1-x$ cars will park there. In this case, the popular region does not extend all the way to the end of the street, hence spot $i+R+1$ must remain open. Since there are $n-i$ total spots to the right of spot $i$ and car $m$ backs into spot $i$, it must be that $1 \leq R \leq n-i-1$. 
            First we select the cars to park in the popular region, which can be done in $\binom{m-1-x}{R}$ ways. 
            Next, spot $i$ must remain empty, thus there are $|\C_{R,R}(k,\ell)|$ ways to park the cars parks in the popular region. 
            Now consider the remaining $m-1-x-R$ cars, which park in spots $i+R+2$ through $n$. 
            Since spot $i+R+1$ must remain open, we use again use the set of contained pullback parking functions, which contributes $|\C_{m-1-x-R,n-R-i-1}(k,\ell)|$ to the count.

            Now we must account for all the possible preferences for car $m$. 
            Since car $m$ backs into spot $i$ and can only back up a maximum of $k$ spots, or, if $R < k$, car $m$ can prefer any spot of the popular region. Hence, the preferences of car $m$ can be $i+1,i+2, \ldots i+\min(R,k)$, which implies that there are $\min(R,k)$ possible preferences for car $m$ so that it backs into spot $i$.
             To conclude, we must sum over all of the possible values of $1\leq R\leq n-i-1$.

            Finally, we sum over all possible $1\leq i \leq n$. Note that $i$ cannot be equal to $n$ in this case, as we must be able to back into $i$, but when $i=n$, this case returns zero, so it is left in the formula in order to simplify the final equation.
            
            Thus, the total count contributed by Subcase~2b is given by 
    \begin{align}\label{case2b}
        &\sum_{i=1}^{n}\sum_{x=0}^{m-1}\sum_{R=1}^{n-i-1}V(i,x,R),
            \end{align}
             where $V(i,x,R)=\binom{m-1}{x}|\PF_{x,i-1}(k,\ell)|\binom{m-1-x}{R}|\C_{R,R}(k,\ell)||\C_{m-1-x-R,n-R-i-1}(k,\ell)|\min(R,k)$. 

            \bigskip
            
            \noindent{\textbf{Case 3:}}
            In this case, car $m$ prefers a spot to the left of spot $i$ and pulls forward into spot $i$.
        
            The popular region, as defined above in Definition \ref{def:popular region}, is located to the left of spot $i$.
            
            We now consider the following two subcases:
            \begin{enumerate}[leftmargin=1in]
            
                \item[\textbf{Subcase 3a:}] Every spot to the left of spot $i$ is occupied, so the popular region extends from spot $1$ to spot $i-1$, hence $R=i-1$.
            
                \item[\textbf{Subcase 3b:}] The popular region does not extend to spot $1$, hence $R \leq i-2$. In this case, car $m$ pulls forward into spot $i$, hence $R\geq k+1$.  Thus $k+1\leq R\leq i-2$.
            \end{enumerate}
            
            We now count the preferences in each of these subcases.

            \noindent    \textbf{Subcase 3a:} Every spot to the left of spot $i$ is occupied, so the popular region extends from spot $1$ to spot $R=i-1$. We illustrate this case in Figure \ref{fig:case3a}.
            
            \begin{figure}[ht]
            \centering\tikzset{every picture/.style={line width=0.75pt}} 

\begin{tikzpicture}[x=0.75pt,y=0.75pt,yscale=-1,xscale=1]

\draw [fill={rgb, 255:red, 0; green, 0; blue, 0 }  ,fill opacity=1 ][line width=2.25]    (124,120.6) -- (550,120.6) ;
\draw    (124,111) -- (124,130.6) ;
\draw    (550,111) -- (550,130.6) ;
\draw  [fill={rgb, 255:red, 255; green, 255; blue, 255 }  ,fill opacity=1 ] (337,120.6) .. controls (337,117.95) and (339.15,115.8) .. (341.8,115.8) .. controls (344.45,115.8) and (346.6,117.95) .. (346.6,120.6) .. controls (346.6,123.25) and (344.45,125.4) .. (341.8,125.4) .. controls (339.15,125.4) and (337,123.25) .. (337,120.6) -- cycle ;
\draw   (336,118) .. controls (336,113.33) and (333.67,111) .. (329,111) -- (237.82,111) .. controls (231.15,111) and (227.82,108.67) .. (227.82,104) .. controls (227.82,108.67) and (224.49,111) .. (217.82,111)(220.82,111) -- (132,111) .. controls (127.33,111) and (125,113.33) .. (125,118) ;
\draw   (549,119) .. controls (549,114.33) and (546.67,112) .. (542,112) -- (455.97,112) .. controls (449.3,112) and (445.97,109.67) .. (445.97,105) .. controls (445.97,109.67) and (442.64,112) .. (435.97,112)(438.97,112) -- (355,112) .. controls (350.33,112) and (348,114.33) .. (348,119) ;
\draw   (125,122) .. controls (125,126.67) and (127.33,129) .. (132,129) -- (219,129) .. controls (225.67,129) and (229,131.33) .. (229,136) .. controls (229,131.33) and (232.33,129) .. (239,129)(236,129) -- (329,129) .. controls (333.67,129) and (336,126.67) .. (336,122) ;
\draw   (348,122) .. controls (348,126.67) and (350.33,129) .. (355,129) -- (436,129) .. controls (442.67,129) and (446,131.33) .. (446,136) .. controls (446,131.33) and (449.33,129) .. (456,129)(453,129) -- (542,129) .. controls (546.67,129) and (549,126.67) .. (549,122) ;
\draw [color={rgb, 255:red, 208; green, 2; blue, 27 }  ,draw opacity=1 ][line width=4.5]    (124,120.6) -- (337,120.6) ;

\draw (166,86.4) node [anchor=north west][inner sep=0.75pt]    {$R=i-1$ \quad \fontfamily{ptm}\selectfont spaces};
\draw (408,86.4) node [anchor=north west][inner sep=0.75pt]    {$n-i$ \quad \fontfamily{ptm}\selectfont spaces};
\draw (176,136.4) node [anchor=north west][inner sep=0.75pt]    {$R=i-1$ \quad \fontfamily{ptm}\selectfont cars};
\draw (369,136.4) node [anchor=north west][inner sep=0.75pt]    {$m-1-( i-1)$ \quad \fontfamily{ptm}\selectfont cars};
\draw (119,136.4) node [anchor=north west][inner sep=0.75pt]    {$1$};
\draw (339,136.4) node [anchor=north west][inner sep=0.75pt]    {$i$};
\draw (545,136.4) node [anchor=north west][inner sep=0.75pt]    {$n$};
\end{tikzpicture}
            \vspace{-.35in}
                \caption{In the figure,  spot $i$ is vacant. The popular region consists of spots $1$ through $i-1$, hence $R=i-1$,  and we highlight the region in red. The region to the right of spot $i$ consists of spots $i+1$ through $n$, which will have the remaining $m-1-(i-1)=m-i$ cars parked in that region.}
            \label{fig:case3a}    
            \end{figure}
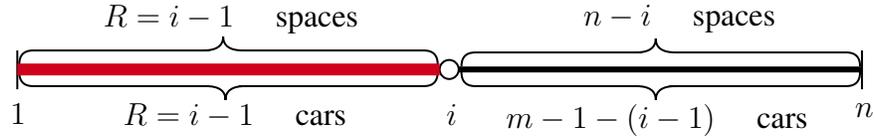
            
            Since in this case, all of the spots from $1$ to $i-1$ are occupied, we can choose $i-1$ cars to park in the popular region and use a pullback parking function to park them. This contributes $\binom{m-1}{i-1}|\PF_{i-1, i-1}(k,\ell)|$ to the count.
            
            We now park the cars to the right of spot $i$. Because we are parking these cars to the right of spot $i$, we must again use the set of contained pullback parking functions, this count is given by $|\C_{m-i,n-i}(k,\ell)|$.
            
            Next, we must consider the possible preferences for car $m$. 
            First, since all spots $1$ to $i-1$ are occupied, the parameter $k$ does not affect the preferences for car $m$. 
            Furthermore, since car $m$ pulls forward into spot $i$, car $m$ can only prefer up to $\ell$ spaces behind spot $i$. 
            However, we must consider that if $R<\ell$, it can only prefer the spots in the popular region.  
            Thus, the number of preferences for car $m$ to pull forward into spot $i$ is given by $\min(R,\ell)=\min(i-1,\ell)$.
            
            We then sum over all possible values of $i$ allowing for the popular region to exist. Hence, $2\leq i\leq n$. However, as before, for sake of simplicity, we index the sum over $1\leq i \leq n$, as the term $i=1$ will contribute zero to the sum.
            
            Thus, Subcase~3a contributes the following to the total count
\begin{align}\label{case3a}
                 \sum_{i=1}^nY(i)\coloneqq\sum_{i=1}^n\binom{m-1}{i-1}\cdot |\PF_{i-1,i-1}(k,\ell)|\cdot |\C_{m-i,n-i}(k,\ell)| \cdot \min(i-1,\ell).
            \end{align}

            \bigskip
            
            \noindent\textbf{Subcase 3b:} The popular region does not extend to spot $1$, Hence $R\leq i-2$. Moreover, as car $m$ pulls forward into spot $i$, we have that $k+1\leq R$. Hence $k+1\leq R\leq i-2$. We illustrate this case in Figure \ref{fig:case3b}.
            
            \begin{figure}[ht]
            
                \centering\tikzset{every picture/.style={line width=0.75pt}} 

\begin{tikzpicture}[x=0.75pt,y=0.75pt,yscale=-1,xscale=1]

\draw [fill={rgb, 255:red, 0; green, 0; blue, 0 }  ,fill opacity=1 ][line width=2.25]    (124,120.6) -- (550,120.6) ;
\draw    (124,111) -- (124,130.6) ;
\draw    (550,111) -- (550,130.6) ;
\draw  [fill={rgb, 255:red, 255; green, 255; blue, 255 }  ,fill opacity=1 ] (337,120.6) .. controls (337,117.95) and (339.15,115.8) .. (341.8,115.8) .. controls (344.45,115.8) and (346.6,117.95) .. (346.6,120.6) .. controls (346.6,123.25) and (344.45,125.4) .. (341.8,125.4) .. controls (339.15,125.4) and (337,123.25) .. (337,120.6) -- cycle ;
\draw   (550,119) .. controls (550,114.33) and (547.67,112) .. (543,112) -- (455.94,112) .. controls (449.27,112) and (445.94,109.67) .. (445.94,105) .. controls (445.94,109.67) and (442.61,112) .. (435.94,112)(438.94,112) -- (354,112) .. controls (349.33,112) and (347,114.33) .. (347,119) ;
\draw   (336,119) .. controls (336,114.33) and (333.67,112) .. (329,112) -- (287.27,112) .. controls (280.6,112) and (277.27,109.67) .. (277.27,105) .. controls (277.27,109.67) and (273.94,112) .. (267.27,112)(270.27,112) -- (228,112) .. controls (223.33,112) and (221,114.33) .. (221,119) ;
\draw   (347,122) .. controls (347,126.67) and (349.33,129) .. (354,129) -- (437.07,129) .. controls (443.74,129) and (447.07,131.33) .. (447.07,136) .. controls (447.07,131.33) and (450.4,129) .. (457.07,129)(454.07,129) -- (543,129) .. controls (547.67,129) and (550,126.67) .. (550,122) ;
\draw   (221,123) .. controls (221,127.67) and (223.33,130) .. (228,130) -- (267.29,130) .. controls (273.96,130) and (277.29,132.33) .. (277.29,137) .. controls (277.29,132.33) and (280.62,130) .. (287.29,130)(284.29,130) -- (329,130) .. controls (333.67,130) and (336,127.67) .. (336,123) ;
\draw [color={rgb, 255:red, 208; green, 2; blue, 27 }  ,draw opacity=1 ][line width=4.5]    (221,120.6) -- (337,120.6) ;
\draw   (124,122) .. controls (124,126.67) and (126.33,129) .. (131,129) -- (158,129) .. controls (164.67,129) and (168,131.33) .. (168,136) .. controls (168,131.33) and (171.33,129) .. (178,129)(175,129) -- (205,129) .. controls (209.67,129) and (212,126.67) .. (212,122) ;
\draw  [fill={rgb, 255:red, 255; green, 255; blue, 255 }  ,fill opacity=1 ] (211.4,120.6) .. controls (211.4,117.95) and (213.55,115.8) .. (216.2,115.8) .. controls (218.85,115.8) and (221,117.95) .. (221,120.6) .. controls (221,123.25) and (218.85,125.4) .. (216.2,125.4) .. controls (213.55,125.4) and (211.4,123.25) .. (211.4,120.6) -- cycle ;
\draw   (211,120) .. controls (211,115.33) and (208.67,113) .. (204,113) -- (177.5,113) .. controls (170.83,113) and (167.5,110.67) .. (167.5,106) .. controls (167.5,110.67) and (164.17,113) .. (157.5,113)(160.5,113) -- (131,113) .. controls (126.33,113) and (124,115.33) .. (124,120) ;

 \draw (147,87) node [anchor=north west][inner sep=-4.75pt]   [align=left] {\begin{minipage}[lt]{29.34pt}\setlength\topsep{0pt}
 \begin{center}
 {\fontfamily{ptm}\selectfont spaces}
\end{center}

\end{minipage}};
\draw (139,76.4) node [anchor=north west][inner sep=-4.75pt]    {$i-R-2$};
\draw (152,137.4) node [anchor=north west][inner sep=0.75pt]    {$x$ \quad \fontfamily{ptm}\selectfont cars};
\draw (256,86.4) node [anchor=north west][inner sep=0.75pt]    {$R$ \quad \fontfamily{ptm}\selectfont spaces};
\draw (257,136.4) node [anchor=north west][inner sep=0.75pt]    {$R$ \quad \fontfamily{ptm}\selectfont cars};
\draw (409,86.4) node [anchor=north west][inner sep=0.75pt]    {$n-i$ \quad \fontfamily{ptm}\selectfont spaces};
\draw (381,136.4) node [anchor=north west][inner sep=0.75pt]    {$m-1-x-R$ \quad \fontfamily{ptm}\selectfont cars};
\draw (119,136.4) node [anchor=north west][inner sep=0.75pt]    {$1$};
\draw (339,135.4) node [anchor=north west][inner sep=0.75pt]    {$i$};
\draw (545,136.4) node [anchor=north west][inner sep=0.75pt]    {$n$};
\end{tikzpicture}
                \vspace{-.35in}
                \caption{In the figure, spot $i$ is vacant. 
                The popular region consists of $R\leq i-2$ cars and is highlighted in red.
                Note spot $i-R-1$ is empty.
                The region to the far left of the street consists of spots $1$ through $i-R-2$, in which $x$ cars park, and $0\leq x\leq i-R-2$. 
                The final region to the right of spot $i$ consists of spots $i+1$ through $n$, in which the remaining $m-1-x-R$ cars park. Note $m-1-x-R\leq n-i$.}
            
                \label{fig:case3b}
                
            \end{figure}
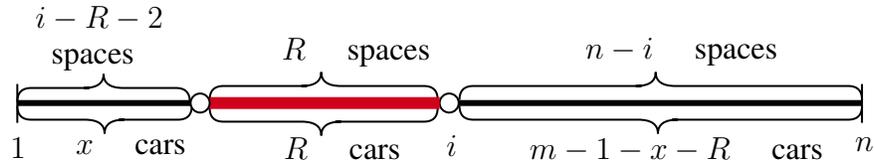
            
            We start by parking the cars to the far left of the street. From $m-1$ cars we choose $x$ cars to park, where $0\leq x\leq m-1$. This can be done in $\binom{m-1}{x}$ ways. The number of spots in this region is $i-R-2$, as we have to subtract the size of the popular region from $i$, and spot $i-R-1$ must remain open. We then park the $x$ cars in the $i-R-2$ spots, which can be done in $|\PF_{x,i-R-2}(k,\ell)|$ ways.
            
            Now we park the cars in the popular region. From the remaining $m-1-x$ cars we select $R$ cars to park in the popular region. This can be done in $\binom{m-1-x}{R}$ ways. Then we park these cars using a contained pullback parking function in order to leave spot $i-R-1$ open. This contributes $|\C_{R,R}(k,\ell)|$ to the count.
            
            We now park all of the cars to the right of spot $i$. There is no choice of cars here since we must use the remaining cars that have yet to be parked. 
            Because we are parking to the right of spot $i$, we use the contained pullback parking function in order to prevent any cars from backing into spot $i$. This contributes $|\C_{m-1-x-R,n-i}(k,\ell)|$ to the count.
            
            Next, we must consider the number of possible preferences for car $m$. In this case, the number of preferences for car $m$ is given by $\min(R-k,\ell)$ because, if $R < \ell$, car $m$ can either prefer any spot from $i-R+k$ to spot $i-1$, or if $R \geq \ell$, it can prefer $\ell$ spaces before spot $i$ in order to pull forward up to $\ell$ spaces into $i$.
            
            We must sum over $3\leq i\leq n$, 
            however as before, we take the sum over $i$ starting at 1, since those terms will contribute zero to the sum. We also must sum
            over $0\leq x\leq m-1$, and over $k+1 \leq R \leq i-2$. 
            Thus, the total count contributed by Subcase~3b is given by
    \begin{align}
            \label{case3b} 
&\hspace{.2in}\sum_{i=1}^n\sum_{x=0}^{m-1}\sum_{R=k+1}^{i-2}W(i,x,R),
            \end{align}
            where $W(i,x,R)=\binom{m-1}{x} |\PF_{x,i-R-2}(k,\ell)| \cdot\binom{m-1-x}{R}|\C_{R,R}(k,\ell)| \cdot |\C_{m-1-x-R,n-i}(k,\ell)|\cdot \min(R-k,\ell)$, and if $a>b$, then $|\PF_{a,b}|=0$.

            The result follows from adding the counts from Equations \eqref{case1}, \eqref{case2a}, \eqref{case2b}, \eqref{case3a}, and~\eqref{case3b}.
        \end{proof}

        We can now specify $\ell=n-1$, in which case $\PF_{m,n}(k,n-1)=\PF_{m,n}(k)$, as the parameter $\ell=n-1$ does not restrict the movement forward. With this in mind, we arrive at the following alternate formula for the number of $k$-Naples $(m,n)$-parking functions.
        
    \begin{corollary}\label{cor:knaples mnpfs}
        
            Let $m\leq n$ be positive integers, and fix $0\leq k \leq n-1$. Then the number of $k$-Naples $(m,n)$-parking functions satisfies the recurrence 
            
            \begin{align*}
                |\PF_{m,n}(k)|=
                &\sum_{i=1}^n\left[X(i)+
                Y(i)+
                \sum_{x=0}^{m-1}\left(Z(i,x)+\sum_{R=1}^{n-i-1}V(i,x,R)+
                \sum_{R=k+1}^{i-2}W(i,x,R)\right)\right],
            \end{align*}
            where \\
            $X(i)=\binom{m-1}{n-i}
                |\PF_{m-1-n+i, i-1}(k)||\C_{n-i,n-i}(k, n-1)| \min(k,n-i)$,
                \\
                $Y(i)= \binom{m-1}{i-1}|\PF_{i-1}(k)| |\C_{m-i,n-i}(k,n-1)| (i-1)$,
                \\   
                $Z(i,x)= \binom{m-1}{x}|\PF_{x,i-1}(k)| |\C_{m-1-x,n-i}(k,n-1)|$,
                \\
                $V(i,x,R)=\binom{m-1}{x}|\PF_{x,i-1}|\binom{m-1-x}{R}|\C_{R,R}(k, n-1)||\C_{m-1-x-R,n-R-i-1}(k, n-1)| \min(R,k)$, 
                \\
                $W(i,x,R)=\binom{m-1}{x} |\PF_{x,i-R-2}(k)| \binom{m-1-x}{R}|\C_{R,R}(k,n-1)| |\C_{m-1-x-R,n-i}(k,n-1)| (R-k)$.
            
        \end{corollary}

By further specifying that $\ell=n-1$ and $m=n$, \Cref{thm:pullback} gives an alternate formula for the number of $k$-Naples parking functions of length $n$, which was first given by 
Christensen et al.~\cite{knaple}. We conclude by stating their theorem.
        \begin{theorem}\cite[Theorem 1.1]{knaple}\label{thm:rational recursion}
            If $k,n \in \NN$ with $0\leq k \leq n-1$, then the number of $k$-Naples parking functions of length $n+1$ is counted recursively by 
            \begin{align*}
                |\PF_{n+1}(k)|=\sum_{i=0}^{n}\binom{n}{i}\min((i+1)+k,n+1)|\PF_{i}(k)|(n-i+1)^{n-i-1}.
            \end{align*}           
        \end{theorem}

    \section{Future directions}\label{sec:future}
    
        We conclude with the following directions for future work. 
        
        
        
        \begin{enumerate}

            \item Our definition of $(k,\ell)$-pullback $(m,n)$-parking functions allows a car whose preference is occupied to back up to $k$ spots, proceed forward up to $\ell$ spots from their preference, and park in the first spot it finds available. 
            Both the parameters $k$ and $\ell$ are set globally. One alternative is for each $i\in [m]$ to have its own nonnegative integer parameters $k_i$ and $\ell_i$, which encode how many spots car $i$ can back up to in its attempt to park, and if all of those spots are taken, then it can check $\ell_i$ spots after its preference, parking in the first available spot it encounters among those $\ell_i$ spots. In this way, whenever $k_i=k$ and $\ell_i=\ell$ for all $i\in[m]$, then one recovers the definition of for $(k,\ell)$-pullback $(m,n)$-parking functions. Also by setting $k_i=0$ for all $i\in[m]$, this would be precisely the definition of interval parking function, which have been studied by Aguilar et al.~\cite{aguilarfraga2024intervalellintervalrationalparking} and Colaric et al.~\cite{bib:Colaric2020IntervalPF}.
            
            \item Subsets of parking functions has led to many interesting enumerative results. 
            For example, it is well-known weakly increasing parking functions are Catalan objects. 
            Reutercrona, Wang, and Whidden~\cite{ICERM_unit_interval} established that the number of weakly increasing unit interval parking functions of length $n$ is given by $2^{n-1}$.
            Fang, Harris, Kamau, and Wang~\cite[Corollary 3.3]{fang2024vacillatingparkingfunctions} established that the number of weakly increasing vacillating parking functions of length $n$ is given by the numerator of the $n$th convergent of the continued fraction of $\sqrt{2}$. 
            It remains an open problem to give enumerations for the weakly increasing subset of $(k,\ell)$-pullback $(m,n)$-parking functions.

            \item We used the technique of counting through permutations to count the number of contained pullback parking functions. We were unable to find closed formulas for these counts. 
            Christensen et al. \cite[Lemma 3.4]{knaple} showed that the cardinality of contained $k$-Naples parking functions of length $n$ is $(n+1)^{n-1}$. It remains an open problem to determine these closed formulas for other specializations of $(k,\ell)$-pullback $(m,n)$-parking functions. 
            
            \item There are many generalizations of parking functions, such as parking functions with different sized cars which are called parking sequences \cite{adeniran2021increasing} and parking assortments \cite{inv_assort, assortcount}. 
            One could recreate our study of pullback parking functions with these generalizations of parking functions. 
        \end{enumerate}

\vspace{24pt}

 \noindent {\bf\Large Acknowledgments}\vspace*{6pt}\\
 The authors were supported in part via award NSF DMS-2150434. Martinez was also supported by the NSF Graduate Research Fellowship Program under Grant No. 2233066. The authors would like to thank the anonymous reviewers.

\bigskip
{\bf \large Contact Information}\\

\begin{tabular}{lcl}
Jennifer Elder	&& Dept. of Computer Science, Mathematics and Physics\\
 jelder8$@$missouriwestern.edu &&  Missouri Western State University\\
 &&St. Joseph, MO 6450, USA\\
 	&& \includegraphics[height=10pt]{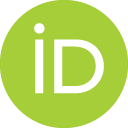}  \url{https://orcid.org/0000-0003-2018-5017}\\
				&& \\
 Pamela E. Harris	&& Department of Mathematical Sciences \\
 peharris$@$uwm.edu && University of Wisconsin-Milwaukee \\
 && Milwaukee, WI 53211, USA\\
 &&\includegraphics[height=10pt]{ORCIDiD_icon128x128.png} \url{https://orcid.org/0000-0002-3049-991X}\\
 					&& \\

 Lybitina Koene	&& Department of Mathematics\\
 lybitinakoene$@$vt.edu &&  Virginia Polytechnic Institute\\
 && Blacksburg, VA 24060, USA \\
 &&\includegraphics[height=10pt]{ORCIDiD_icon128x128.png} \url{https://orcid.org/0009-0000-8234-6686}\\
					&& \\

 Ilana Lavene&& Department of Mathematical Sciences  \\
 ilanalavene$@$gmail.com && University of Delaware\\
 && Newark, DE 19711, USA \\
 &&\includegraphics[height=10pt]{ORCIDiD_icon128x128.png} \url{https://orcid.org/0009-0005-7980-7425}\\
 					&& \\

 Lucy Martinez&& Department of Mathematics, Rutgers University\\
 lucy.martinez$@$rutgers.edu && Piscataway, NJ 08854, USA \\
 &&\includegraphics[height=10pt]{ORCIDiD_icon128x128.png} \url{https://orcid.org/0000-0003-3144-4053}\\
 					&& \\

 Molly Oldham&& Department of Mathematical Sciences\\
 molly.oldham343$@$gmail.com &&  University of Delaware\\
 && Newark, DE 19711, USA \\
 &&\includegraphics[height=10pt]{ORCIDiD_icon128x128.png} \url{https://orcid.org/0009-0007-7708-1426}\\
 					&& \\

 \end{tabular}


\begin{thebibliography}{99}

\bibitem{adeniran2021increasing} Ayomikun Adeniran and Catherine H. Yan, {\it On increasing and invariant parking sequence}, Australas. J. Combin. 79 (2021) 167--182.

\bibitem{aguilarfraga2024intervalellintervalrationalparking} Tomás Aguilar-Fraga, Jennifer Elder, Rebecca E. Garcia, Kimberly P. Hadaway, Pamela E. Harris, Kimberly J. Harry, Imhotep B. Hogan, Jakeyl Johnson, Jan Kretschmann, Kobe Lawson-Chavanu, J. Carlos Martínez Mori, Casandra D. Monroe, Daniel Quiñonez, Dirk Tolson III, and Dwight Anderson Williams II, {\it Interval and $\ell$-interval rational parking functions}, Discrete Math. Theor. Comput. Sci. 26(1) (2024) Paper No. 10, 29.

\bibitem{BaumgardnerHonorsContract} Alyson Baumgardner, {\it The Naples parking function}, Florida Gulf Coast University (2019).

\bibitem{inv_assort} Douglas M. Chen, Pamela E. Harris, J. Carlos Mart\'{i}nez Mori,  Eric J. Pab{\'o}n-Cancel, and Gabriel Sargent, {\it Permutation invariant parking assortments}, Enumer. Comb. Appl. 4(1) (2024) Paper No. S2R4, 25.


\bibitem{knaple} Alex Christensen,  Pamela E. Harris, Zakiya Jones, Marissa Loving, Andr\'{e}s Ramos Rodr\'{\i}guez, Joseph Rennie, and Gordon Rojas Kirby, {\it A generalization of parking functions allowing backward movement}, Electron. J. Combin. 27 (2020).


\bibitem{bib:Colaric2020IntervalPF} Emma Colaric, Ryan DeMuse, Jeremy L.  Martin, and Mei Yin, {\it Interval parking functions}, Adv. in Appl. Math. 123 (2021) Paper No. 102129, 17.

\bibitem{countingthroughperms} Laura Colmenarejo, Pamela E. Harris, Zakiya Jones, Christo Keller, Andr\'es Ramos Rodr\'iguez, Eunice Sukarto, and Andr\'es R. Vindas-Mel\'endez, {\it Counting {$k$}-{N}aples parking functions through permutations and the {$k$}-{N}aples area statistic}, Enumer. Comb. Appl. 1 (2021) Paper No. S2R11, 16.


\bibitem{fang2024vacillatingparkingfunctions} Bruce Fang, Pamela E. Harris, Brian M. Kamau, and David Wang, {\it Vacillating parking functions}, arXiv:2402.02538 (2024). To appear in J. Comb.


\bibitem{assortcount} Spencer J. Franks, Pamela E. Harris, Kimberly J. Harry, Jan Kretschmann, and Megan Vance, {\it Counting parking sequences and parking assortments through permutations}, Enumer. Comb. Appl. 4 (2024) Paper No. S2R2, 10.


\bibitem{github} Lybitina Koene, Ilana Lavene, and Molly Oldham, {\it GitHub repository for pullback parking function experiments}, available at \url{https://github.com/lybitinakoene/Pullback-Parking-Functions.git}.


\bibitem{konheim1966occupancy} Alan G. Konheim and Benjamin Weiss, {\it An occupancy discipline and applications}, SIAM J. Appl. Math. 14 (1966) 1266--1274.


\bibitem{ICERM_unit_interval} Eva Reutercrona and Yuxuan (Susan) Wang and Juliet Whidden, {\it Parking functions with fixed ascent and descent sets}, Summer@ICERM, End of Summer Report (2022).

\bibitem{Spiro} Sam Spiro, {\it Subset parking functions}, J. Integer Seq. 22 (2019) Art. 19.7.3, 15.

\bibitem{StanleyECVol2} Richard P. Stanley, {\it Enumerative Combinatorics. {V}ol. 2}, Cambridge Studies in Advanced Mathematics, vol. 62, Cambridge University Press, Cambridge (1999).
\end{thebibliography}
\end{document}